\documentclass[12pt]{amsart}
\usepackage{color,psfrag,epsfig,url,hhline
}
\numberwithin{figure}{section}
\numberwithin{table}{section}
\newtheorem{theorem}{Theorem}[section]

\newtheorem{lemma}[theorem]{Lemma}

\theoremstyle{definition}
\newtheorem{definition}[theorem]{Definition}
\newtheorem{example}[theorem]{Example}

\theoremstyle{remark}
\newtheorem{remark}[theorem]{Remark}
\numberwithin{equation}{section}
\newfont{\tap}{tap scaled 650}

\def \N{{\mathbb N}}

\def \g{{\mathfrak g}}
\def \h{{\mathfrak h}}
\def \Z{{\mathbb Z}}

\def \O{{\mathcal O}}

\def \Q{{\mathbb Q}}

\def \A{{\mathcal A}}
\def \[{[ }
\def \]{] }

\def\F{\mathcal{F}}

\def\t{\widetilde}
\def\u{\underline}
\def\s{\sigma}

\def\h{\widehat}

\def\p{\partial}

\def\x{\mathbf{x}}
\def\y{\mathbf{y}}

\def\cc{\mathbf{c}}
\def\gg{\mathbf{g}}
\def\deg{\mathrm{deg}\,}

\def\PP{{\mathbb P}}
\def\ZZ{{\mathbb Z}}
\def\TT{{\mathbb T}}

\def \B{{\mathcal B}}

\begin{document}

\title{Bases for cluster algebras from orbifolds}

\author{Anna Felikson and Pavel Tumarkin}
\address{Department of Mathematical Sciences, Durham University, Science Laboratories, South Road, Durham, DH1 3LE, UK}
\email{anna.felikson@durham.ac.uk, pavel.tumarkin@durham.ac.uk}
\thanks{AF was partially supported by EPSRC grant EP/N005457/1}

\begin{abstract}
We generalize the construction of the bracelet and bangle bases defined in~\cite{MSW2} and the band basis defined in~\cite{T2} to cluster algebras arising from orbifolds. We prove that the bracelet bases are positive, and the bracelet basis for the affine cluster algebra of type $C_n^{(1)}$ is atomic. We also show that cluster monomial bases of all skew-symmetrizable cluster algebras of finite type are atomic. 

\end{abstract}

\maketitle
\setcounter{tocdepth}{1}
\tableofcontents

\section{Introduction}

Cluster algebras were introduced by Fomin and Zelevinsky~\cite{FZ1} in the effort to understand a construction of canonical bases by Lusztig~\cite{L} and Kashiwara~\cite{K}. A cluster algebra is a commutative ring with a distinguished set of generators called {\em cluster variables}. Cluster variables are grouped into overlapping finite collections of the same cardinality called {\em clusters} connected by local transition rules which are determined by a skew-symmetrizable {\em exchange matrix} associated with each cluster, see Section~\ref{cluster} for precise definitions.

One of the central problems in cluster algebras theory is a construction of good bases. It was conjectured in~\cite{FZ1} that these bases should contain cluster monomials, i.e. all products of cluster variables belonging to every single cluster. Linear independence of cluster monomials in skew-symmetric case was proved by Cerulli Irelli, Keller, Labardini-Fragoso and Plamondon in~\cite{CKLP}, for a general skew-symmetrizable case linear independence was recently proved by Gross, Hacking, Keel and Kontsevich in~\cite{GHKK}. In the finite type cluster monomials themselves form a basis (Caldero and Keller~\cite{CK}).

Bases containing cluster monomials were constructed for various types of cluster algebras. These include ones by Sherman and Zelevinsky~\cite{SZ} (rank two affine type), Cerulli~Irelli~\cite{C1} (affine type $\t A_{2}$), Ding, Xiao and Xu~\cite{DXX} (affine type), Dupont~\cite{D1,D3}, Geiss, Leclerc and Schr\"oer~\cite{GLS}, Plamondon~\cite{P} ({\em generic} bases for acyclic types), Lee, Li and Zelevinsky ({\em greedy} bases in rank two algebras).       

In~\cite{MSW2} Musiker, Schiffler and Williams constructed two types of bases  ({\em bangle} basis $\B^\circ$ and {\em bracelet} basis $\B$) for cluster algebras originating from unpunctured surfaces~\cite{FST,FT,FG0}. A {\em band} basis (we call it $\B^\s$) was introduced by D.~Thurston in~\cite{T2}. All the three bases are parametrized by collections of mutually non-intersecting arcs and closed loops, and all their elements are {\em positive}, i.e. the expansion of any basis element in any cluster is a Laurent monomial with non-negative coefficients. 

In the present paper, we extend the construction of all the three bases to cluster algebras originating from orbifolds.

\begin{theorem}
\label{bases-thm}
Let $\A$ be a cluster algebra with principal coefficients constructed by an unpunctured orbifold with at least two boundary marked points. Then $\B$, $\B^\s$ and $\B^\circ$ are bases of $\A$.

\end{theorem}

Our main tools are the tropical duality by Nakanishi and Zelevinsky~\cite{NZ}, and the theory of unfoldings developed in~\cite{FeSTu2,FeSTu3}. The notion of an unfolding was introduced by Zelevinsky, it provides a reduction of problems on (certain) skew-symmetrizable cluster algebras to appropriate skew-symmetric ones. In our case, unfoldings allow us to treat cluster algebras from orbifolds using the results known for cluster algebras from surfaces, in particular, to use the results of~\cite{MSW2}.  

The tropical duality provides a relation between  $\gg$-vectors and $\cc$-vectors of a cluster algebra with principal coefficients. For cluster algebras originating from surfaces and orbifolds, the $\cc$-vectors have an explicit geometric meaning: they can be viewed as collections of shear coordinates of specially constructed laminations~\cite{FGl,FG,FT}. To prove linear independence of the constructed bases, $\gg$-vectors are used in~\cite{MSW2}. We use tropical duality together with results of~\cite{FeSTu3} to supply $\gg$-vectors with geometric meaning: they can be read off from shear coordinates of some laminations on a dual orbifold.         

\medskip
 We also consider positivity properties of the constructed bases. A basis of cluster algebra is called {\it positive} if it has positive structure constants. Positivity of $\B$ for surfaces was conjectured in~\cite{FG0} and proved in~\cite{T2}. We extend this result to the orbifold case.

\setcounter{section}{9}
\setcounter{theorem}{1}
\begin{theorem}
The bracelet basis $\B$ is positive.

\end{theorem}
The bangle basis $\B^\circ$ is not positive as shown in~\cite{T}, the band basis $\B^\s$  is conjectured to be positive~\cite{T2}.

A basis is called {\it atomic} if it satisfies the following property: non-negative linear combinations of basis elements are exactly those elements of the algebra whose Laurent expansion is positive in any cluster. If it exists, atomic basis is unique~\cite{SZ}. It is proved in~\cite{C} (see also~\cite{CL}) that the cluster monomial bases of skew-symmetric cluster algebras are atomic. We extend this result to the full generality.

\setcounter{section}{10}
\setcounter{theorem}{1} 

\begin{theorem}
Cluster monomial bases of skew-symmetrizable cluster algebras of finite type are atomic.

\end{theorem}

The bracelet basis for surfaces was conjectured to be atomic in~\cite{MSW2}. In particular, the atomic basis for $\widetilde A_{p,q}$ constructed in~\cite{DT} is precisely the bracelet basis for an unpunctured annulus with $p$ and $q$ points at the boundary components. We use unfoldings to prove the following result.

\setcounter{theorem}{2}

\begin{theorem}
The bracelet basis of the cluster algebra of the affine type $C_n^{(1)}$ is atomic.

\end{theorem}

\setcounter{section}{1}

Recently, Gross, Hacking, Keel and Kontsevich~\cite{GHKK} constructed a canonical positive {\em theta} basis for every cluster algebra of geometric type. 
An interesting question is the relation between the theta bases and the bases constructed in~\cite{MSW2} and the present paper. The theta basis cannot coincide with the bangle basis since the latter is not positive. According to~\cite{CGMMRSW}, theta bases of cluster algebras of rank two coincide with greedy bases. In particular, for affine algebra $\t A_{1,1}$ the theta basis is precisely the bracelet basis (and not the band basis). Is this always the case, i.e. does the bracelet basis coincide with the theta basis for all surfaces and orbifolds?  

We also note that all the bases in~\cite{MSW2} were constructed for unpunctured surfaces with at least two boundary marked points. It was shown in~\cite{CLS} that the results of~\cite{MSW2} also hold for unpunctured surfaces with a single boundary marked point. Following~\cite{CLS}, one can show that the results of the current paper can also be extended to orbifolds with a single boundary marked point, the proof to appear in~\cite{CT}.

\medskip

\noindent
The paper is organized as follows.

In preparatory Sections~\ref{cluster} and~\ref{surf} we recall basic notions on cluster algebras
and remind the construction of cluster algebras from triangulated bordered surfaces and orbifolds.

In Section~\ref{unf} we define an orbifold unfolding as a ramified covering branching in the orbifold points only. 
As it was mentioned above, making use of unfoldings is one of our main tools in this paper. However, this only works when all curves in consideration
{\it ``lift well''} in the unfolding.
We introduce the notion of a curve which lifts well in a given unfolding and show that for each curve of our interest there exists an unfolding where the curve
lifts well. This technical statement is crucial for our proofs.

In Section~\ref{skein} we build the skein theory for the orbifold case. Section~\ref{bases} is devoted to the construction of bracelet, band and bangle bases $\B$, $\B^\s$ and $\B^\circ$ for the cluster algebras from orbifolds.
We use tropical duality to prove that these sets are bases in Sections~\ref{reps} and~\ref{ind}.

In Section~\ref{pos} we discuss positivity property of the bracelet basis $\B$. Finally, in Section~\ref{atomic} we show that the bracelet basis for affine cluster algebra of type $C_n^{(1)}$ is atomic. We also prove that cluster monomial bases of cluster algebras of types $B_n$, $C_n$ and $F_4$ are atomic.  
\medskip

\subsection*{Acknowledgements}
We would like to thank G.~Musiker, M.~Shapiro and D.~Thurston for numerous stimulating discussions, and G.~Muller and S.~Stella for explaining us details of their paper~\cite{CGMMRSW}.
We are grateful to D.~Thurston for his lecture course ``Curves on surfaces'' in Berkeley, 2012. The paper was partially written in MSRI during the program on cluster algebras. We thank the organizers for the opportunity to participate  in the program, and the institute for a nice working atmosphere.

\section{Basics on cluster algebras}
\label{cluster}

\noindent
We briefly remind  some definitions and notions introduced by Fomin and Zelevinsky in~\cite{FZ1} and~\cite{FZ4}.

\subsection{Cluster algebras}
An integer $n\times n$ matrix $B$ is called \emph{skew-symmetrizable} if there exists an
integer diagonal $n\times n$ matrix $D=diag(d_1,\dots,d_n)$,
such that the product $BD$ is a skew-symmetric matrix, i.e., $b_{ij}d_j=-b_{ji}d_i$.

Let $\PP$ be \emph{a tropical semifield } $\mathrm{Trop\,}(u_1,\dots,u_m)$, i.e. an abelian group freely generated by elements $u_1,\dots,u_m$  with commutative multiplication $\cdot$ and equipped with addition $\oplus$
defined as $$\prod_j u_j^{a_j}\oplus \prod_j u_j^{b_j}=\prod_j u_j^{min(a_j,b_j)} . $$
 The multiplicative group of $\PP$ is \emph{a coefficient group} of cluster algebra. 
$\ZZ\PP$  is the integer group ring. Define $\F$ as the field of rational functions in $n$ independent
variables with coefficients in the field of fractions of $\ZZ\PP$.
$\F$ is called {\it an ambient field}.

\begin{definition}
\emph{A seed} is a triple $(\x,\y,B)$, where
\begin{itemize}
\item
$\x=\{x_1,\dots,x_n\}$ is a collection of algebraically independent rational functions of $n$ variables which generates $\F$ over the field of fractions of $\ZZ\PP$;
\item
$\y=\{y_1,\dots,y_n\}$, is an $n$-tuple of elements of $\PP$ called a \emph{coefficient tuple} of cluster $\x$;
\item
$B$ is a skew-symmetrizable integer matrix (\emph{exchange matrix}).
\end{itemize}
The part $\x$ of a seed $(\x,\y,B)$ is called \emph{cluster}, elements $x_i$ are called \emph{cluster variables}, the part $\y$ is called {\it coefficient tuple}.

\end{definition}

We denote $[x]_+=max(x,0)$.

\begin{definition}[seed mutation]
For any $k$, $1\le k\le n$ we define \emph{the mutation} of seed $(\x,\y,B)$ in direction $k$
as a new seed $(\x',\y',B')$ in the following way:
\begin{equation}\label{eq:MatrixMutation}
b'_{ij}=\left\{
           \begin{array}{ll}
             -b_{ij}, & \hbox{ if } i=k \hbox{ or } j=k; \\
             b_{ij}+\frac{|b_{ik}|b_{kj}+b_{ik}|b_{kj}|}{2}, & \hbox{ otherwise.}\\
           \end{array}
         \right.
\end{equation}
\begin{equation}\label{eq:CoeffMutation}
  y'_i=\left\{
           \begin{array}{ll}
             y_k^{-1}, & \hbox{ if } i=k; \\
             y_iy_k^{[b_{ki}]_+}(y_k\oplus 1)^{-b_{ki}} & \hbox{ if } i\ne k.
           \end{array}
         \right.
\end{equation}

\begin{equation}\label{eq:ClusterMutation}
x'_i=\left\{
           \begin{array}{ll}
             x_i, & \hbox{ if } i\ne k; \\
             \frac{y_k\prod x_j^{[b_{jk}]_+}+\prod x_j^{[-b_{jk}]_+}}{(y_k\oplus 1)x_k}, & \hbox{ otherwise.}
           \end{array}
         \right.
\end{equation}

\end{definition}

\noindent
We write $(\x',\y',B')=\mu_k\left((\x,\y,B)\right)$.
Notice that $\mu_k(\mu_k(\x,\y,B))=(\x,\y,B)$.
Two seeds are called \emph{mutation-equivalent}
if one is obtained from the other by a sequence of seed mutations.
Similarly one says  that two clusters or two exchange matrices are \emph{mutation-equivalent}.

Notice that exchange matrix mutation~(\ref{eq:MatrixMutation}) depends only on the exchange matrix itself.
The collection of all matrices mutation-equivalent to a given matrix $B$ is called the \emph{mutation class} of $B$.

Following~\cite{FZ4}, we define a {\it cluster pattern} by assigning to every vertex $t$ of an $n$-regular tree $\TT_n$ with edges labeled by $\{1,\dots,n\}$ a seed $\Sigma_t=(\x_t, \y_t, B_t)$, such that two vertices $t$ and $t'$ are joined by an edge labeled by $k$ if and only if $ \Sigma_{t'}=\mu_k(\Sigma_t)$. We denote the elements of $\Sigma_t$ by
$$\x_t=(x_{1;t},\dots,x_{n,t}),\quad \y_t=(y_{1;t},\dots,y_{n,t}), \quad B_t=(b_{ij}^t)$$

For any skew-symmetrizable matrix $B$ an \emph{initial seed} $\Sigma_{t_0}=(\x_{t_0},\y_{t_0},B_{t_0})=(\x,\y,B)$ is a collection
$(\{x_1,\dots,x_n\}$, $\{y_1,\dots,y_n\},B)$, where $B$ is the \emph{initial exchange matrix}, $\x=\{x_1,\dots,x_n\}$ is the \emph{initial cluster}, $\y=\{y_1,\dots,y_n\}$ is the \emph{initial coefficient tuple}.

{\it Cluster algebra} $\A(B)$ associated with the skew-sym\-met\-ri\-zab\-le $n\times n$ matrix $B$ is a subalgebra of $\Q(x_1,\dots,x_n)$ generated by all cluster variables of the clusters mutation-equivalent to the initial seed $(\x,\y,B)$.

Cluster algebra $\A(B)$ is called \emph{of finite type} if it contains only finitely many
cluster variables. In other words, all clusters mutation-equivalent to initial cluster contain
only finitely many distinct cluster variables in total.

\begin{definition}\label{def:FinMutType} A cluster algebra  is said to be \emph{of finite mutation type} if it has finitely many exchange matrices.
\end{definition}

An {\it extended exchange matrix} $\tilde B=(b_{ij})$ is an $(n+m)\times n$ matrix whose upper $n\times n$ matrix is $B$ and lower $m\times n$ part $B^0$ encodes the coefficient tuple using the formula
$$
y_j=\prod\limits_{i=1}^m u_i^{b_{n+i,j}}.
$$ 
In terms of the matrix $\tilde B$, the exchange relation~\ref{eq:ClusterMutation} rewrites as
$$
x_k'=x_k^{-1}\left( \prod\limits_{j=1}^n x_j^{[b_{jk}]_+}\prod\limits_{j=1}^m u_j^{[b_{(n+j)k}]_+} +
  \prod\limits_{j=1}^n x_j^{[-b_{jk}]_+}\prod\limits_{j=1}^m u_j^{[-b_{(n+j)k}]_+} \right).
$$

\subsection{Cluster algebras with principle coefficients}
\label{principal}

A cluster algebra has {\it principle coefficients at a seed} $\Sigma_{t_0}=(\x,\y,B)$ if $\PP=\mathrm{Trop\,}(y_1,\dots,y_n)$.

In other words,  a cluster algebra $\A$ with principle coefficients is associated to a $2n\times n$ matrix $\tilde B$,
whose upper part is $B$ and the lower (coefficient) part $B^0$ is $n\times n $ identity matrix.

A cluster algebra with principle coefficients associated to the matrix $B$ is denoted by $\A_\bullet(B)$.

Consider $\Z^n$-grading on $\A_\bullet(B)$ defined by $$\deg(x_i)=\mathbf e_i, \qquad \deg(y_j)=-\mathbf b_j^0,$$
where $\mathbf e_1,\dots,\mathbf e_n$ is the standard basis vectors  in $\Z^n$ ($\mathbf e_i$ has 1 at $i$-th position and 0 at other places) and $\mathbf b^0_j=\sum_i b_{ij}^0\mathbf e_i$ is the $j$-th column of $B^0$.

It is shown in~\cite{FZ4} that the Laurent expression of any cluster variable in any cluster of $\A_\bullet(B)$
is homogeneous with respect to this $\Z^n$-grading. For each cluster variable $x$, the 
 $\gg$-{\it{vector }} $\gg(x)$ with respect to the seed $(\x,\y,B)$ is the multi-degree of the Laurent expansion of $x$ with respect to $(\x,\y,B)$. 

Given the initial seed $\Sigma_{t_0}=(\x,\y,B)$ of $\A_\bullet(B)$ and a seed $\Sigma_{t}=(\x_{t},\y_{t},B_{t})$, we denote by $C^{B;t_0}_t$ the lower part $B_t^0$ of the exchange matrix  $\tilde B_t$ (here $B$ is the exchange matrix at $t_0$). The columns $\cc_{1;t},\dots,\cc_{n,t}$ of $C^{B;t_0}_t$ are called $\cc$-{\it{vectors }} at $\Sigma_{t}$, the matrix $C^{B;t_0}_t$ is called $\cc$-{\it{matrix }} at $\Sigma_{t}$. Similarly, we may denote by  $G^{B;t_0}_t$ the matrix composed of $\gg$-vectors $\gg_{1;t},\dots,\gg_{n,t}$ at $\Sigma_{t}$, i.e. $\gg_{i;t}=\gg(x_{i;t})$. 

According to~\cite[(1.13)]{NZ}, there is a duality between $\cc$- and $\gg$-vectors: $(G_t^{B;t_0})^T=C_{t_0}^{B_t^T;t}$.


\section{Cluster algebras from surfaces and orbifolds}
\label{surf}

In this section we remind the construction of cluster algebras arising from triangulated surface~\cite{FST,FT} and its generalization to the orbifold case~\cite{FeSTu3}.

\begin{remark}
In the most part of this paper we work with unpunctured surfaces/orbifolds, so we will not review tagged triangulations
and will ignore self-folded triangles
(for the details see~\cite{FST}).

\end{remark} 

\subsection{Cluster algebras from surfaces}
Let $S$ be a bordered surface with a finite number of marked points, and with at least one marked point at each boundary component.

A non-boundary marked point is called a {\it puncture}. In the most part of this paper we will have no punctures.

An {\it arc} $\gamma$ is a non-self-intersecting curve with two ends at marked points (may be coinciding) such that
\begin{itemize}
\item[-] except for its endpoints, $\gamma$ is disjoint from the marked points and the boundary;
\item[-] $\gamma$ does not cut out a monogon not containing any marked points;
\item[-] $\gamma$ is not homotopic to a boundary segment.
\end{itemize}

A {\it triangulation} of $S$ is a maximal collection of mutually non-homotopic disjoint arcs (two arcs are allowed to have a common vertex).  

Given a triangulation $T$ of $S$ one builds a skew-symmetric {\it signed adjacency matrix} $B=(b_{ij})$ as follows: 
\begin{itemize}
\item the rows and columns of $(b_{ij})$ correspond to the arcs in $T$; 
\item let $b_{ij}=\sum_{\triangle\in T} b_{ij}^\triangle$, where for every triangle $\triangle\in T$ the number $b_{ij}^\triangle$ is defined in the following way:
$$
b_{ij}^\triangle=
\begin{cases}
1\quad\ \textrm{if the arcs $i$ and $j$ belong to $\triangle$ and follow in $\triangle$ in a clockwise order},\\
$-1$\quad \textrm{if the arcs $i$ and $j$ belong to $\triangle$ and follow in $\triangle$ in a counter-clockwise order},\\
0\quad\ \textrm{otherwise}. 
\end{cases}
$$
\end{itemize}

A mutation $\mu_k$ of the matrix $B$ corresponds to a {\it flip} of the triangulation in the $k$-th edge, where 
a flip is a move replacing a diagonal of a quadrilateral by another diagonal, or, more generally,
replacing an arc $\gamma\in T$ by a unique other arc disjoint from $T\setminus \gamma$.
More precisely, the signed adjacency matrix of the flipped triangulation is $\mu_k(B)$.

Introducing a variable $x_\gamma$ for each $\gamma\in T$ and using the matrix $B$ as initial 
exchange matrix one can build a cluster algebra $\mathcal A(S)$. 

It is shown in~\cite{FST,FT} that 
\begin{itemize}
\item the algebra $\mathcal A(S)$ does not depend on the initial triangulation chosen in $S$;
\item cluster variables of $\mathcal A(S)$ are in bijection with arcs on $S$;
\item several cluster variables belong to one cluster if and only if the corresponding arcs belong to one triangulation.

\end{itemize}

Given a hyperbolic metric on $S$, one can think of marked points as cusps (and thus, the triangles in the triangulation can be thought as ideal triangles). Choose in addition a horocycle around each marked point, and for each arc $\gamma$ define $l(\gamma)$
as the length of the part of $\gamma$ staying away from the horocycles centred in both ends.
A {\it lambda length} of $\gamma$ (as defined in~\cite{Pe}) is $\lambda_\gamma=e^{l(\gamma)/2}$.

These functions $\lambda_\gamma$ are subject to {\it Ptolemy relation} under the flips: $$\lambda_\gamma\lambda_{\gamma'}=
\lambda_\alpha\lambda_\sigma+\lambda_\beta\lambda_\rho,$$ where $\gamma$ and $\gamma'$ are two diagonals of the quadrilateral with sides $\alpha, \beta, \sigma, \rho$. The Ptolemy relation coincides with the exchange relation for cluster variables in $\mathcal A(S)$, so that cluster variables may be interpreted as lambda lengths, 
and we will write $x_\gamma=\lambda_\gamma$ (as a function of lambda lengths of the arcs in the initial triangulation).

Cluster algebras arising from surfaces are of finite mutation type, moreover, it is shown in~\cite{FeSTu1}
that all but finitely skew-symmetric cluster algebras of finite mutation type of rank $n>2 $ are cluster algebras arising from surfaces.

\subsection{Cluster algebras from orbifolds}

To obtain (almost all) skew-symmetrizable cluster algebras of finite mutation type one introduces cluster algebras from triangulated orbifolds~\cite{FeSTu3}. 

 By an {\it orbifold} $\O$ we mean a triple $\O=(S,M,Q)$, where $S$ is a bordered surface with a finite set of marked points $M$, and $Q$
is a finite (non-empty) set of special points called {\it orbifold points}, $M\cap Q=\emptyset$. Some marked points may belong to $\p S$ (moreover, every boundary component must contain at least one marked point; the interior marked points are still called {\it punctures}), while
all orbifold points are interior points of $S$ (later on, as we will supply the orbifold with a metric, the orbifold points will have cone angle $\pi$). By the  {\it boundary} $\p \O$ we mean $\p S$.

An {\it arc} $\gamma$ in $\O$ is a curve in $\O$ considered up to relative isotopy (of $\O\setminus \{M\cup Q\}$) modulo endpoints such that
\begin{itemize}
\item one of the following holds:
\begin{itemize}
\item either both endpoints of $\gamma$ belong to $M$ (and then $\gamma$ is an {\it ordinary arc})
\item or one endpoint belongs to $M$ and another belongs to $Q$ (then $\gamma$ is called a {\it pending arc});
\end{itemize}
\item $\gamma$ has no self-intersections, except that its endpoints may coincide;
\item except for the endpoints, $\gamma$ and $M\cup Q\cup \p \O$ are disjoint;
\item if $\gamma$ cuts out a monogon then this monogon contains either a point of $M$
or at least two points of $Q$;
\item $\gamma$ is not homotopic to a boundary segment.

\end{itemize}

Note that we do not allow both endpoints of $\gamma$ to be in $Q$.

Two arcs $\gamma$ and $\gamma'$ are {\it compatible} if the following two conditions hold:
\begin{itemize}
\item they do not intersect in the interior of $\O$;
\item if both $\gamma$ and $\gamma'$ are pending arcs, then the ends of $\gamma$ and $\gamma'$ that are orbifold points do not coincide { (i.e., two pending arcs may share a marked point, but neither an ordinary point nor a orbifold point)}.

\end{itemize}

A {\it triangulation} of $\O$ is a maximal collection  of distinct pairwise compatible arcs.
The arcs of a triangulation cut $\O$ into {\it triangles}. 

\subsubsection{Weights}
Each orbifold point of $\O$ comes with a {\it weight} $w=1/2$ or $w=2$. 
A pending arc incident to an orbifold point with weight $w$ is assigned with the same weight $w$.
An ordinary arc is assigned with the weight $w=1$. Denote by $w_i$ the weight of $i$-th arc and let
$$
\varepsilon=\left\{
           \begin{array}{ll}
             2 & \hbox{ if $\O$ contains at least one orbifold point of weight $w=1/2$,  }  \\
             1 & \hbox{ otherwise.}
           \end{array}
         \right.
$$

Given a triangulation $T$ of $\O$ one builds a skew-symmetrizable {\it signed adjacency matrix} $B=(b_{ij})$ as follows: 
\begin{itemize}
\item the rows and columns of $B$ correspond to the arcs in $T$ (both ordinary and pending arcs are considered here); 
\item let $b_{ij}=\varepsilon\sum_{\triangle\in T} b_{ij}^\triangle$, where for every triangle $\triangle\in T$ the number $b_{ij}^\triangle$ is defined in the following way:
$$
b_{ij}^\triangle=
\begin{cases}
w_i\quad \textrm{if the arcs $i$ and $j$ belong to $\triangle$ and follow in $\triangle$ in a clockwise order},\\
$-$w_j\ \ \textrm{if the arcs $i$ and $j$ belong to $\triangle$ and follow in $\triangle$ in a counter-clockwise order},\\
0\quad\ \ \textrm{otherwise}. 
\end{cases}
$$

\end{itemize}

Introducing a variable $x_\gamma$ for each $\gamma\in T$ and using the matrix $B=(b_{ij})$ one can build a cluster algebra
$\mathcal A(\O)$. 

The geometric realization of $\mathcal A(\O)$ is obtained via {\it associated orbifold} denoted by $\h\O$.
In the associated orbifold, we substitute all weight $2$ orbifold points by {\it special marked points} and introduce
a hyperbolic metric on $\h\O$ such that
\begin{itemize}
\item the marked points (including special marked points) are cuspidal points on $\h\O$; 
\item the remaining orbifold points (i.e., orbifold points of weight $1/2$) are orbifold points with cone angle
around each of them equal to $\pi$;
\item the triangles in $T$ not incident to special marked points are ideal hyperbolic triangles,  the pending arcs of weight $1/2$ understood as two halves of a side of an ideal triangle glued to each other; the pending arcs of weight $2$ are called {\it double arcs} (in fact, these can be understood as a pair of tagged arcs which are never mutated separately);
\item each special marked point $p$ is endowed with a fixed self-conjugated horocycle, i.e. a horocycle  $h_p$ such that the hyperbolic length of $h_p$ equals $1$.

\end{itemize}

\begin{remark}
In this paper (except for Section~\ref{atomic}) we consider orbifolds without orbifold points of weight $2$. In that case $\h\O$ coincides with $\O$ (where $\O$ is understood as a union of ideal hyperbolic triangles), so we will omit the word ``associated'' and will call $\h\O$ simply ``an orbifold''. Also, by technical reasons, we assume that 
$\h\O$ has at least two marked points.
\end{remark}

Now, one can choose in addition an horocycle around each non-special marked point, 
and for every arc $\gamma$ define $l(\gamma)$
as the length of the part of $\gamma$ staying away from the horocycles centered in both ends (for a pending arc $\gamma$
the length $l(\gamma)$ is understood as the length of the round trip from a horocycle to the orbifold point and back).
The {\it lambda length} of $\gamma$ is defined as $\lambda_\gamma=e^{l(\gamma)/2}$.

It is shown in~\cite{FeSTu3} that lambda lengths of arcs satisfy the exchange relations of the cluster algebra 
$\mathcal \A(\O)$, so that they can be interpreted as geometric realizations of cluster variables. 

\subsection{Coefficients and laminations}

In the case of cluster algebras from surfaces or orbifolds, the coefficients can be visualized using laminations
(or more precisely, shear coordinates of laminations).

By a {\it lamination} on $\h\O$ we mean an integral unbounded measured lamination, i.e. 
a finite collection of non-self-intersecting and mutually disjoint curves on $\h\O$ modulo isotopy;
every curve here is either a closed curve or a curve each of whose ends is of one of the following three types:
\begin{itemize}
\item[-] an unmarked point of the boundary of $\h\O$;
\item[-] a spiral around a puncture contained in $M$ (either clockwise or counter-clockwise);
\item[-] an orbifold point.

\end{itemize}

Also, the following is not allowed:
\begin{itemize}
\item a curve that bounds an unpunctured disc or a disc containing a unique marked, special marked or orbifold point;
\item a curve with two endpoints on the boundary of $\h\O$ isotopic to a piece of boundary containing no marked points
or a single marked point;
\item two curves starting at the same orbifold point (or two ends of the same curve starting at the same orbifold point);
\item curve spiralling in or starting at any special marked point.
\end{itemize}

Given a lamination $L$ on a surface, the {\it shear} coordinates of $L$ with respect to a given triangulation $T$ 
(containing no self-folded triangles) were introduced by  W.~Thurston~\cite{Th} and  can be computed as follows. 
For each arc $\gamma$ of $T$ the corresponding
{\it shear coordinate} of $L$ with respect to the triangulation $T$, denoted by $b_\gamma (T,L)$, is defined as a sum of
contributions from all intersections of curves in $L$ with the arc $\gamma$. Such an intersection contributes $+1$ (resp, -1)
to  $b_\gamma (T,L)$ if the corresponding segment of the curve in $L$ cuts through the quadrilateral surrounding $\gamma$ as
shown in Fig.~\ref{lam} on the left (resp, on the right).

\begin{figure}[!h]
\begin{center}
\psfrag{g}{\scriptsize $\gamma$}
\psfrag{1}{\small $+1$}
\psfrag{-1}{\small $-1$}
\epsfig{file=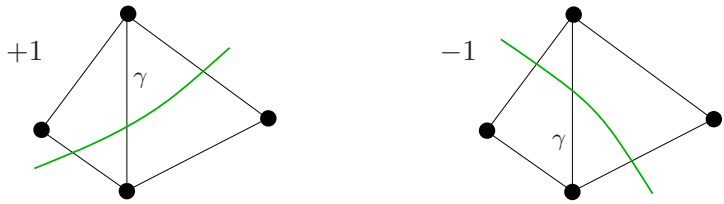,width=0.6\linewidth}
\caption{Defining the shear coordinate  $b_\gamma (T,L)$ on surfaces.}
\label{lam}
\end{center}
\end{figure}

The construction is extended to the case of arbitrary (tagged) triangulation of a surface in~\cite{FT}, and to the case of an orbifold in~\cite{FeSTu3}. In particular, the following theorem holds.

\begin{theorem}[\cite{FeSTu3}, Theorem~6.7]
\label{ex-uniq}
Let $\h\O$ be an associated orbifold.
For a fixed  
triangulation $T$ of $\h\O$, the map
$$ L \to (b_\gamma(T,L))_{\gamma \in T}$$
is a bijection between laminations on $\h\O$ and $\Z^n$.

\end{theorem}

A {\it multi-lamination}
is a finite set of laminations $\mathbf L=(L_{1},\dots,L_m)$.

Given a multi-lamination  $\mathbf L=(L_{1},\dots,L_m)$ and an initial triangulation $T_{t_0}$, compose an extended signed adjacency matrix $\tilde B$ of the signed adjacency matrix $B$ and the $m\times n$ matrix of the shear coordinates of $\{L_i\}$ in $T_{t_0}$. It appears (see~\cite{FT,FeSTu3}) that the transformation of the matrix $\tilde B$  under a flip coincides with its transformation under the corresponding mutation.  This implies that, given a multi-lamination $\mathbf L$, we can consider shear coordinates of $\mathbf L$ in the triangulation $T_t$ as coefficients in the seed $\Sigma_{t}$.

In particular, we can keep track of principle coefficients, for which we need {\it elementary laminations}.
For each arc $\gamma$ of a triangulation $T_{t_0}$ one can choose the lamination $L_\gamma$ such that  $b_\gamma (T_{t_0},L_\gamma)=1$
and  $b_\sigma (T_{t_0},L_\gamma)=0$ for $\sigma\ne \gamma$ (it does exist and is unique by Theorem~\ref{ex-uniq}, in the unpunctured case it may be obtained from the arc $\gamma$ by shifting all (non-orbifold) ends clockwise, see Fig.~\ref{elemlam} for examples).

\begin{figure}[!h]
\begin{center}
\epsfig{file=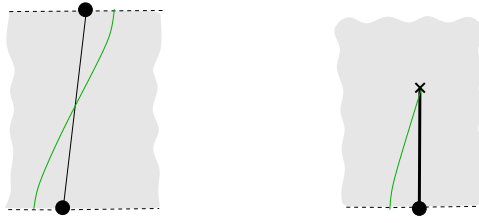,width=0.4\linewidth}
\caption{Elementary laminations on an unpunctured surface/orbifold}
\label{elemlam}
\end{center}
\end{figure}

Composing a multi-lamination $\mathbf L=(L_{\gamma_1},\dots,L_{\gamma_n})$ of elementary laminations for all curves $\{\gamma_1,\dots,\gamma_n\}$ of $T_{t_0}$, we obtain an extended signed adjacency matrix with the coefficient part $B^0=I$. Now we can identify the $\cc$-matrix $C^{B;t_0}_t$ at $\Sigma_{t}$ with the matrix of shear coordinates of $\mathbf L$ in the triangulation $T_t$. 

\section{Unfoldings and curves on orbifolds}
\label{unf}

\subsection{Orbifold unfoldings}
The idea of {\it unfolding} was suggested by Andrei Zelevinsky. Roughly speaking, an unfolding of a skew-symmetrizable matrix $B$ is a skew-symmetric matrix $C$ such that the properties of the cluster algebra $\mathcal A(B)$ can be read off the properties of the cluster algebra $\mathcal A(C)$ (see~\cite{FeSTu2,FeSTu3} for details). It turns out that not every skew-symmetrizable algebra has an unfolding, however, in finite mutation type an unfolding almost always exists and proves to be useful. For a cluster algebra from an orbifold an unfolding can be provided by an algebra from a surface, where the surface is a ramified covering of the orbifold branching in the orbifold points only~\cite{FeSTu3} and with all ramification indices equal to two. In this paper, we need a notion of a {\it partial unfolding} which satisfies some weaker assumptions.

\begin{definition}
Given an orbifold  $\h\O$, a ramified covering $\h S$ of $\h\O$ branching in the orbifold points only with all ramification indices equal to two is called a {\it partial unfolding}. In addition, $S=\h S$ is called an  {\it unfolding} if it contains no orbifold points (or, equivalently, if every orbifold point is a branch point). We supply $\h S$ with a hyperbolic metric such that the covering map is a local isometry everywhere except for the ramification points.

By the {\it degree} of a partial unfolding we mean the degree of the covering. In this paper we will only consider partial unfoldings of degree $2^k$.
\end{definition}

By a {\it complete lift} of a curve $\gamma\subset \h\O$ to a partial unfolding $\h S$ we mean the union of all lifts of $\gamma$.

Given a partial unfolding $\h S$ of $\h\O$, each triangulation on $\h\O$ lifts to a triangulation on $\h S$.  Moreover, since the covering map is a local isometry, it preserves the lengths of arcs, and thus it preserves lambda lengths (one can note that, due to the definition of a length of pending arc, the covering map also preserves lengths of pending arcs). Therefore, each cluster $\mathbf x$ in $\mathcal A(\h\O)$ lifts to a cluster $\bar{\mathbf x} $ in $\mathcal A(\h S)$ (each cluster variable of $\bar{\mathbf x}$ is a lift of some cluster variable of $\mathbf x $). These lifts agree with mutations~\cite{FeSTu3} (see also Remark~\ref{spec}).

For any function $f$ in variables of the cluster $\bar{\mathbf x} $ by the
{\it specialization} $f|_{\h\O}$ we mean the function in variables of $\mathbf x $ where each cluster variable of $\bar{\mathbf x} $ 
is substituted by its image in $\mathbf x $.

\begin{remark}
\label{spec}
Let $S$ be a degree $n$ partial unfolding of $\h\O$, let $\mathbf x $ be a cluster on $\h\O$ and $\bar{\mathbf x} $ be its lift on $\h S$. Let $\gamma$ be an arc or a pending arc on $\h\O$ and $\bar \gamma$ be a connected component of its lift, $x_\gamma\in\mathbf x$. Write $x_\gamma$ as the Laurent expansion in the cluster  $\mathbf x $, and $x_{\bar \gamma}$ as the Laurent expansions in the cluster $\bar {\mathbf x} $.

Then the definitions above imply $x_{\bar\gamma }|_{\h\O}=x_\gamma$.

\end{remark}

\subsection{Curves on orbifolds}

Throughout the paper, all curves on surfaces and orbifolds are considered up to isotopy. 
In particular, every intersection of a family of curves is thought as a simple transversal intersection of two curves. Further, we assume a number of self-intersections of any curve to be finite. 
By the length of a curve we mean the length of the geodesic representative of the isotopy class.

By a {\it regular point} of an orbifold $\h\O$ we mean any interior point except for orbifold points.

A curve is called {\it separating} if it cuts the surface or orbifold into more than one connected components.
Otherwise, it is called {\it non-separating}.

We consider several types of curves: with two ends in marked points (possibly coinciding), with one end in a marked point and another in an orbifold point, with two ends in orbifold points (possibly coinciding) or a closed curve. We call these curves an {\it ordinary curve},  a {\it pending curve}, a {\it semi-closed curve} and  
a {\it closed curve} respectively. If the curves of these types are in addition non-self-intersecting we call them
 an {\it arc},  a {\it pending arc}, a {\it semi-closed loop} and  
a {\it closed loop} (these definitions agree with definitions of arcs and pending arcs given above). Graphically, we denote orbifold points by crosses, their lifts to (partial) unfoldings by small circles and the marked points by bold circles. The curves incident to an orbifold point 
are drawn thick. We also say that pending curves and semi-closed curves are {\it thick} curves and all other types of curves are {\it thin}. We summarize the definitions above in Table~\ref{curve-def}.

\begin{table}
\caption{Types of curves on orbifolds}
\label{curve-def}
\begin{tabular}{|c|c|c|c|}
\hline
& curves & non-self-intersecting curves&\\
\hline  
\raisebox{-4pt}{\epsfig{file=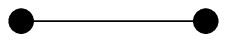,width=0.2\linewidth}}& \raisebox{10pt}{ordinary curve} & \raisebox{10pt}{arc}&\raisebox{-10pt}{thin}\\
\cline{1-3}
\raisebox{-4pt}{
\epsfig{file=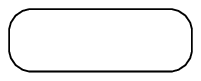,width=0.2\linewidth}}& \raisebox{10pt}{closed curve}& \raisebox{10pt}{closed loop}&\\
\hline
\raisebox{-4pt}{\epsfig{file=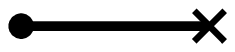,width=0.2\linewidth}}& \raisebox{10pt}{pending curve}& \raisebox{10pt}{pending arc}&\raisebox{-10pt}{thick}\\
\cline{1-3}
\raisebox{-4pt}{
\epsfig{file=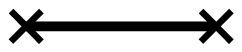,width=0.2\linewidth}}& \raisebox{10pt}{semi-closed curve}& \raisebox{10pt}{semi-closed loop}&\\
\hline
\end{tabular}
\end{table}

\begin{remark}
The term ``pending arc'' was introduced in~\cite{Ch} and~\cite{ChM}. 
The term ``semi-closed loop'' is introduced since in a degree $2$ partial unfolding branching in the both ends of the curve it lifts to a closed loop.

\end{remark}

\begin{remark}
A curve connecting an orbifold point to itself will be considered as self-intersecting. By closed curves we mean ones not contractible to a point or an orbifold point.

\end{remark}

\subsection{Curves in unfoldings}

\begin{definition}
\label{def well}
Let $\h S$ be a degree $n$ partial unfolding of $\h\O$. 
A closed curve $\gamma\in \h\O$ {\it lifts well} to $\h S$ if the complete lift of $\gamma$ in $\h S$ consists of $n$ disjoint closed curves.
An ordinary curve, a pending curve and a semi-closed curve are always said to lift well.

\end{definition}

\begin{example}
We show an example of an unfolding and a closed curve which does not lift well. Let $\h\O$ be an orbifold of genus $g>0$ with some marked points and precisely two orbifold points, and let $\gamma$ be a closed curve on $\h\O$. Take a path $\alpha$ connecting two orbifold points and crossing $\gamma$ exactly once. We build a degree $2$ partial unfolding $\h S$ as follows. Cut the orbifold along $\alpha$, take two copies of the obtained orbifold and glue them together to obtain a connected surface. Then the curve $\gamma$ lifts to one closed curve of length $2l$ (where $l$ is the length of $\gamma$), so it does not lift well.
\end{example}

The following lemma plays the key role in the proof of skein relations for orbifold (see Section~\ref{skein-proof}).

\begin{lemma}
\label{lifts well}
Let $\h\O$ be an unpunctured orbifold.
Let  $\gamma\subset \h\O$ be a closed curve. 
Then there exists an unfolding $S$ such that $\gamma$ lifts well in $S$.

\end{lemma}

\begin{proof}
In~\cite{FeSTu3} we used the following four ways to construct degree two (or four) covering branched in orbifold points only:
\begin{itemize}
\item[(1)] if the number of orbifold points is even, group orbifold points in pairs, cut the orbifold along mutually non-intersecting semi-closed loops connecting paired orbifold points, take two copies and glue them along the cuts to obtain one surface;
\item[(2)] if the number of orbifold points is odd and bigger than one, apply (1) to just one pair of orbifold points; this will result in an orbifold with an even number of orbifold points, so (1) can be applied;
\item[(3)] if there is exactly one orbifold point and the boundary is non-empty, cut the orbifold along a pending arc ending at the boundary, then glue together two copies to obtain a surface;
\item[(4)] if there is exactly one orbifold point, the boundary is empty, and the genus is positive, construct an unramified degree two covering by cutting the orbifold along a closed non-separating loop and gluing two copies together to double the number of orbifold points, then apply (1). 
\end{itemize}
Note that there is no required covering of closed orbifolds of genus zero with a unique orbifold point. 

We now want to combine these to construct a covering where  $\gamma$ lifts well.

After cutting $\h\O$ along  $\gamma$ the following connected components $\h S_i$ may appear:
\begin{itemize}
\item[(i)] $\h O_i$ contains no orbifold points;
\item[(ii)] $\h O_i$ contains orbifold points and at least one boundary component of $\h\O$;
\item[(iii)] $\h O_i$ contains orbifold points, at least one handle and no boundary components of $\h\O$;
\item[(iv)] $\h O_i$ is a disc with orbifold points. 
\end{itemize}

Now our aim is to choose mutually non-intersecting  non-self-intersecting paths we cut the orbifold along in such a way that $\gamma$ intersects each of them even number of times. This will guarantee that $\gamma$ lifts well in the covering. We need to choose such paths to be incident to every orbifold point exactly once, with all other ends being at the boundary marked points. Application of (1) -- (4) will then provide a required covering.

First, we connect pairs of orbifold points in the same connected components and make cuts along these paths. This results in at most one orbifold point in every connected component. Next, in all components of type (ii) with one orbifold point, connect this point to a boundary marked point, and make cuts along these paths. 
In all components of type (iii) with one orbifold point we choose a closed non-separating loop and cut the orbifold along all these loops. Note that all of these cuts do not change connectedness of the connected components of $\h\O\setminus\gamma$, so this can be done in any order. 

We are left to consider discs with one orbifold point each (we do not count holes obtained from the cuts already made). First, we make the following observation: if we take three orbifold points, then we can choose two of them that can be joined by a path intersecting $\gamma$ at an even number of points. Thus, if there are more than two discs with one orbifold point each, we can join some pair of orbifold points by a path intersecting $\gamma$ an even number of times, make a cut along this path, and then repeat this procedure until at most two orbifold points are left. Denote these points by $A$ and $B$. Clearly, $A$ and $B$ cannot be connected by a path  intersecting $\gamma$ at an even number of points.

Consider any boundary marked point $p$.  Then exactly one of  $A$ and $B$ (say, $B$) can be connected to $p$ by a path  intersecting $\gamma$ at an even number of points. Connect it to $p$ by such a path and cut along it. We are left with one orbifold point $A$ only. 

Now we take two copies of obtained orbifold and glue them along the cuts. We obtain a partial unfolding $\h S$ of degree two in which $\gamma$ lifts well. Assume that $\h S$ is not a surface, otherwise the lemma is already proved. 

Denote by $\bar\gamma$ the complete lift of  $\gamma$. Then any connected component $\h S_i$ of  $\h S\setminus\bar\gamma$ containing orbifold points is of one of the two types: either $\h S_i$ contains exactly two orbifold points (these come from connected components of type (iii) with odd number of orbifold points), or $\h S_i$ is a disc containing one of the lifts of point $A$ (denote these two orbifold points by $A_1$ and $A_2$). The former ones can be treated as before by cutting along a path between the two points. Let us prove the following statement: 
\medskip

\paragraph{{\bf Claim.}}{\em Assume that $\h S$ does not coincide with $\h\O$. Then there exists a path $\bar\alpha$ connecting  $A_1$ to $A_2$ and intersecting each of the two lifts of $\gamma$ at an even number of points.}
\medskip

To prove the claim, we look for a closed non-contractible path $\alpha\subset\h\O$ through $A$ satisfying the two conditions:
\begin{itemize}
\item[-] $\alpha$ intersects $\gamma$  in an even number of points;
\item[-] $\alpha$ intersects the cuts made on $\h\O$ in order to construct $\h S$ an odd number of times. 
\end{itemize}
Then any lift of $\alpha$ will connect $A_1$ to $A_2$ and satisfy the assumptions of the claim.

Since $\h S$ does not coincide with $\h\O$, at least one cut was made at the first step. If there was a cut connecting some orbifold point $A'$ with another orbifold point or a boundary point, then we can take as $\alpha$  a loop around $A'$ based at $A$. If the only cuts were closed loops in connected components of type (iii), then we can choose one such component $\h\O_i$ and take any closed loop in $\h\O_i$ intersecting just one of the cuts exactly once, and then connect this loop to $A$ obtaining a loop  based at $A$.

Thus, the claim is proved. Note that we may assume $\bar\alpha$ does not intersect other cuts we have made on $\h S$. Cutting $\h S$ along $\bar\alpha$ and taking two copies of obtained orbifold, we obtain a degree four unfolding $S$ of $\h\O$ in which $\gamma$ lifts well. 

We are left to consider the case when $\h S$ coincides with  $\h\O$. This, in particular, means that $A$ is the only orbifold point of $\h\O$. We consider three cases: either $\h\O$ is a sphere with a unique boundary component, or it is a sphere with several boundary components, or it has positive genus.

In the former case there are no closed curves, so there is nothing to consider. In the latter case we can always find a non-separating closed curve $\beta$ intersecting $\gamma$ an even number of times: take two homologically non-equivalent non-separating simple closed loops, and if each of them has odd number of intersections with $\gamma$, take their sum. Now we cut along this curve, glue two copies and obtain an orbifold with two orbifold points, so by the construction above there is exists an unfolding of degree four or eight in which $\gamma$ lifts well.

Last, assume that $\h\O$ has genus zero and at least two boundary components. Take marked points $p_1$ and $p_2$  on different boundary components. If $A$ can be joined with $p_1$ or $p_2$ by a path intersecting $\gamma$ at an even number of points, then we proceed as before to obtain a degree two unfolding. Otherwise, $p_1$ and $p_2$ can be connected by a path intersecting   $\gamma$ an even number of times. This path is non-separating, so we can construct a degree two covering with two orbifold points, and then an unfolding of degree four or eight. 

\end{proof}

\section{Skein relations on orbifolds}
\label{skein}

\subsection{Skein relations on surfaces}
We first recall the definition of skein relations for curves on surfaces (see e.g.~\cite{MW}).

We define a {\it multicurve} to be a finite collection of curves with finite number of intersections among them. Given a multicurve $C$ containing two curves intersecting at point $p$ (or a self-intersecting curve), we define two new multicurves $C_+$ and $C_-$ by replacing the crossing as in Fig.~\ref{res}
(we call the multicurves  $C_+$ and $C_-$ a {\it resolution of the intersection $p$ in $C$} and write $R_p(C)=C_++C_-$ understanding the right-hand side as a formal sum).

\begin{figure}[!h]
\begin{center}
\psfrag{C}{$C$}
\psfrag{C-}{$C_+$}
\psfrag{C+}{$C_-$}
\epsfig{file=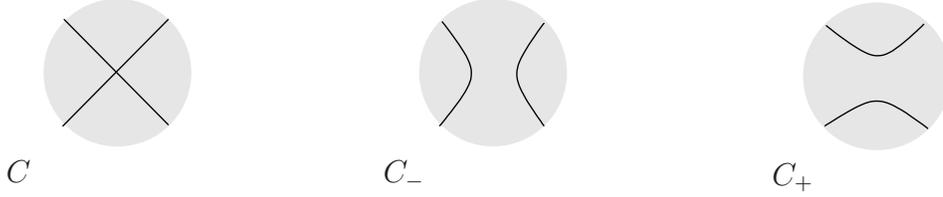,width=0.79\linewidth}
\caption{Resolution of the intersection on a surface: $R_p(C)=C_++C_-$}
\label{res}
\end{center}
\end{figure}

Recall that given an arc $\alpha$, one denotes by $x_{\alpha}$ the cluster variable corresponding to $\alpha$ 
(or its Laurent expansion in a given cluster  $(x_{\gamma_1},\dots,x_{\gamma_n})$). 
For the case of the closed or self-intersecting curve $\alpha$ the Laurent polynomial $x_{\alpha}$ is defined 
in~\cite{MSW,MW} (via explicit formula in terms of a snake graph).

For a multicurve $C=\{\alpha_1,\dots,\alpha_k\}$ one defines an element $x_C$ of the corresponding cluster algebra as a product 
$x_C=x_{\alpha_1}\cdot\dots\cdot x_{\alpha_k}$.

For a finite formal sum of multicurves $\sum k_iC_i$ the Laurent polynomial $x_{\sum k_iC_i}$ is defined by $x_{\sum k_iC_i}=\sum k_i x_{C_i}$.

According to~\cite{MW}, the Laurent expression $x_C$ for $C$ in the cluster $(x_{\gamma_1},\dots,x_{\gamma_n})$ can be expressed via Laurent expressions $x_{C_+}$ and  $x_{C_-}$ for $C_+$ and $C_-$ as follows.


\begin{lemma}[\cite{MW}, Propositions~6.4--6.6]
\label{skein-surface}
Given a multicurve $C$ on an unpunctured surface, the following equality holds: 
$$x_C=Y_+x_{C_+}+Y_-x_{C_-},$$
where $Y_+$ and $Y_-$ are monomials in variables $y_i$. 

\end{lemma}

\begin{remark}
For the case of a punctured surface similar relations are obtained in~\cite[Propositions~6.12--6.14]{MW}
under the condition that the initial triangulation contains no self-folded triangles. 

\end{remark}

The powers of $y_i$ in the monomials $Y_+$ and $Y_-$ can be expressed via intersection numbers of the multicurves $C,C_+,C_-$ with elementary laminations $L_i$ for the initial triangulation $({\gamma_1},\dots,{\gamma_n})$ (see~\cite{MW} for details). If one of the connected components in a multicurve $C_+$ and $C_-$ occurs to be a contractible loop, this component is substituted by a multiple $-2$.



In the coefficient-free case the skein relations from Lemma~\ref{skein-surface} simplify to  
$$x_C=x_{C_+}+x_{C_-}.$$

\subsection{Skein relations on orbifolds}

We want to show counterparts of formulae from~\cite[Propositions~6.4--6.6]{MW} (and, in particular, of Lemma~\ref{skein-surface}) for orbifolds. 
For this, we need to redefine the intersection numbers, the multicurves $C_+$ and $C_-$, and the Laurent polynomials associated to curves.

One of the ways to obtain skein relations for orbifolds is the following. Consider an unfolding $S$ of the orbifold $\h\O$.
If a multicurve $C$ on  $\h\O$ has an intersection in a point $p\in\h\O$, we can consider the lifts of the curves in $C$ through $p$ to the unfolding $S$, resolve the intersection on the surface  by applying the skein relation provided in~\cite{MW}, and then look at the image of the obtained multicurves on $\h\O$. This approach (in the coefficient-free case) leads to the result in Table~\ref{skein-table}. The verifications are straightforward in the assumption that each component of the multicurve $C$ lifts well in $S$. The proof of the general case requires some preparation.

\subsection{Resolutions of curves on orbifolds}

\begin{definition}
\label{def skein}
Let $C$ be a multicurve. Suppose that $p$ is an intersection point of $C$ (an intersection of some 
curves $\gamma,\gamma'\in C$ or a self-intersection of some curve $\gamma\in C$).
The {\it resolution of $C$ in $p$} will be denoted by $R_p(C)$ and defined as shown in Table~\ref{skein-table}. 

More precisely, in most cases
 $$R_p(C)=C_++C_-+2C_=,$$ where $C_+$, $C_-$ and $C_=$ are as in Table~\ref{skein-table} (we understand the right-hand side as a formal sum). 
When $C_=$ is absent in  Table~\ref{skein-table}, we define  $R_p(C)=C_++C_-$.

\begin{remark}
Note that our definition of resolution differs a bit from one used, e.g., in~\cite{T}: we consider the whole sum but not a single summand.
\end{remark}

\begin{center}
\begin{table}[!h]
\caption{Resolution of curves on orbifolds}
\label{skein-table}
\begin{tabular}{|c||ccc|c|}
\hline
&&&&\\
$C$&$C_+$&$C_-$&$C_=$&resolution $R_{p}(C)$\\
&&&&\\
\hline
\hline
&&&&\\
\epsfig{file=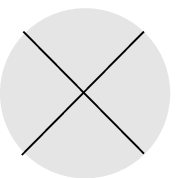,width=0.07\linewidth}&\epsfig{file=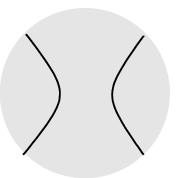,width=0.07\linewidth}&
\epsfig{file=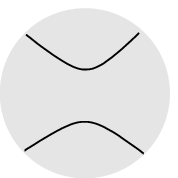,width=0.07\linewidth}& \raisebox{13pt}{absent}&\raisebox{13pt}{$C_++C_-$}\\
&&&&\\
\hline
&&&&\\
\phantom{aa} \epsfig{file=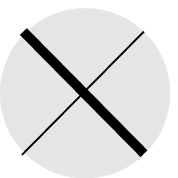,width=0.07\linewidth}\phantom{aa} &\phantom{aa} \epsfig{file=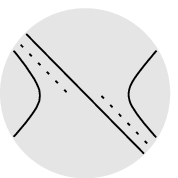,width=0.07\linewidth}\phantom{aa} &\phantom{aa} 
\epsfig{file=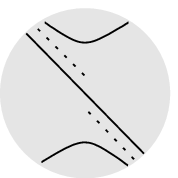,width=0.07\linewidth}\phantom{aa} & \raisebox{13pt}{absent}&\raisebox{13pt}{$C_++C_-$} \\
&&&&\\
\hline
&&&&\\
\phantom{aa} \epsfig{file=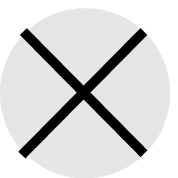,width=0.07\linewidth}\phantom{aa}  &\phantom{aa}  \epsfig{file=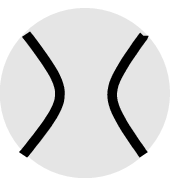,width=0.07\linewidth}\phantom{aa}  &\phantom{aa} 
\epsfig{file=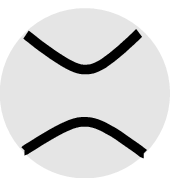,width=0.07\linewidth}\phantom{aa}  &\phantom{aa} \epsfig{file=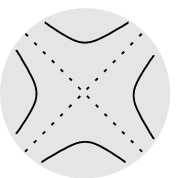,width=0.07\linewidth}\phantom{aa}&\raisebox{13pt}{$C_++C_-+2C_=$} \\
&&&&\\
\hline
&&&&\\
\psfrag{g1}{\scriptsize $\gamma_1$}
\psfrag{g2}{\scriptsize $\gamma_2$}
\phantom{aa} \raisebox{-4mm}{\epsfig{file=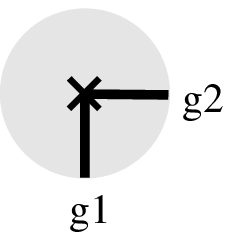,width=0.098\linewidth}} &\phantom{aa} \epsfig{file=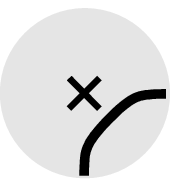,width=0.07\linewidth}\phantom{aa} &\phantom{aa} 
\epsfig{file=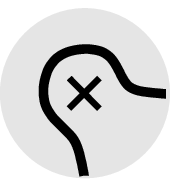,width=0.07\linewidth}\phantom{aa} &  \raisebox{13pt}{absent}
&
\raisebox{13pt}{
\begin{tabular}{ll}
$C_++C_-$ & \quad if $\gamma_1\ne \gamma_2 \vphantom{\int\limits_A}$\\
$C_++C_--2$ & \quad if $\gamma_1=\gamma_2$\\ 
\end{tabular}
}
\\
&&&&\\
\hline
\end{tabular}
\end{table}
\end{center}

The entries of Table~\ref{skein-table} should be interpreted in the following way:

\begin{itemize}
\item
If a thick curve $\gamma$ in $C_+$ or $C_-$ is not incident to any orbifold point, then it is interpreted as two thin curves isotopic
to $\gamma$ (see Example~\ref{skein-ex}).
\item 
The multicurve $C_=$ in the third row builds as follows:
\begin{itemize}
\item[-] for each of the four directions (shown by dashed rays) two thin curves follow this direction;
\item[-] at the end of each dashed ray one has either a marked point or an orbifold point (since a thick curve is never closed);
\item[-] for a marked endpoint, both thin curves meet at this point;
\item[-] for an orbifold endpoint, the two curves coming to this point make one (thin) curve going around the orbifold point (see Example~\ref{skein-ex});
\end{itemize}
\item 
The multicurves $C_+$ and $C_-$ in the second row build similarly to $C_=$:
\begin{itemize}
\item[-] each of two directions shown by a dashed ray is followed by two thin curves, the curves either meet at a marked
endpoint or form one curve travelling around an orbifold endpoint.
\end{itemize}
\item 
Each contractible closed loop $\gamma$ in a multi-curve $C_\pm$ is substituted by a multiple $-2$
(so that if $C=\{\gamma_1,\dots,\gamma_k\}$ and $\gamma_1,\dots,\gamma_m$ are contractible closed loops,
then $C=(-2)^mC_1$, where $C_1=\{\gamma_{m+1},\dots,\gamma_k\}$).
\item 
Each semi-closed loop with coinciding endpoints cutting out a disc without orbifold points is substituted by a multiple $2$.
\item 
Each arc with coinciding endpoints cutting out a disc without orbifold points is substituted by a multiple $0$.
\item 
Each closed loop cutting out a disc with a unique orbifold point is substituted by a multiple $0$.

\end{itemize}

\end{definition}

\begin{example}
\label{skein-ex}
Here are three examples of resolutions: 
$$
R_p(\raisebox{-12pt}{\epsfig{file=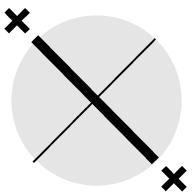,width=0.07\linewidth}})\  = \
\raisebox{-12pt}{\epsfig{file=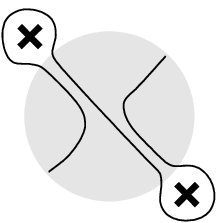,width=0.073\linewidth}} \ + \
\raisebox{-12pt}{\epsfig{file=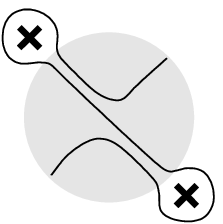,width=0.073\linewidth}}
\vphantom{++2\raisebox{-12pt}{\epsfig{file=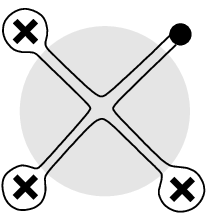,width=0.073\linewidth}} }
$$
\smallskip
$$
R_p(\raisebox{-12pt}{\epsfig{file=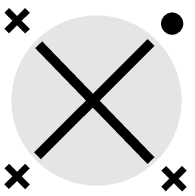,width=0.07\linewidth}})\  = \
\raisebox{-12pt}{\epsfig{file=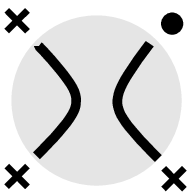,width=0.07\linewidth}} \ + \
\raisebox{-12pt}{\epsfig{file=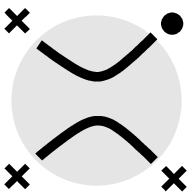,width=0.07\linewidth}} \ + \
2\raisebox{-12pt}{\epsfig{file=./pic/r28.eps,width=0.073\linewidth}} 
$$
\medskip
$$
R_p(\raisebox{-12pt}{\epsfig{file=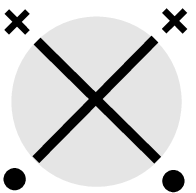,width=0.07\linewidth}})\  = \
\raisebox{-12pt}{\epsfig{file=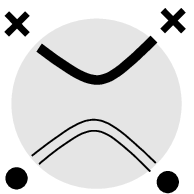,width=0.07\linewidth}} \ + \
\raisebox{-12pt}{\epsfig{file=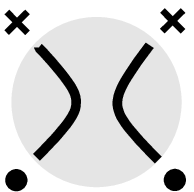,width=0.07\linewidth}} \ + \
2\raisebox{-12pt}{\epsfig{file=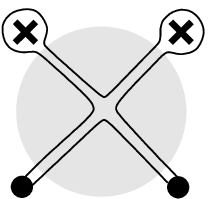,width=0.073\linewidth}} 
$$
\medskip
\end{example}

\begin{remark}
\label{factor of res}
The reason to define the multicurves $C_+,C_-,C_=$ as in Table~\ref{skein-table} is the following. Let $C$ be a multi-curve in $\h\O$, $p$ is an intersection point of $C$. Denote by $\h C$ the complete collection of curves of $C$ containing $p$. Let $S$ be an unfolding of $\h\O$ such that all components of $\h C$ lift well. Let $\bar p\in S$ be any lift of $p$, and let $\bar C$ be the complete collection of lifts of curves of $\h C$ containing $\bar p$.

Then the image of the resolution $R_{\bar p}(\bar C)$ on $S$ under the covering map $S\to\h\O$ coincides with the resolution  $R_{p}(\h C)$ on $\h\O$.
\end{remark}

\begin{remark}
\label{skein-th}
There is another way to describe skein relations on orbifolds (suggested by Dylan Thurston). As before, assume that $\gamma$ and $\gamma'$ are the curves in a multicurve $C$ intersecting at a point $p$. Then one can proceed in the following way.

- If $\gamma$ is a thin curve, leave it intact. 

- If $\gamma$ is a pending curve, substitute it by an ordinary curve with two ends in the marked end  of $\gamma$ going around the orbifold end of $\gamma$ (see Fig.~\ref{deform}(a)). 

- If $\gamma$ is a semi-closed curve, substitute it by an ordinary closed curve going around $\gamma$ (see  Fig.~\ref{deform}(b)).

- Do the same for $\gamma'$. 

- Apply usual skein relations for all crossings of the images of $\gamma$ and $\gamma'$ (as we have thin curves only).

- Substitute curves with marked end around an orbifold point by pending curves, closed curves around two orbifold points by semi-closed curves, contractible closed curves around an orbifold point by $0$.

It is easy to see that the result of the procedure above is exactly the same as one described in Table~\ref{skein-table}. This can be explained as follows: if we consider the curves as geodesics on $\h\O$, every substitution above is a small deformation of the geodesic. 

\end{remark}

\begin{figure}[!h]
\begin{center}
\psfrag{a}{\small (a)}
\psfrag{b}{\small (b)}
\epsfig{file=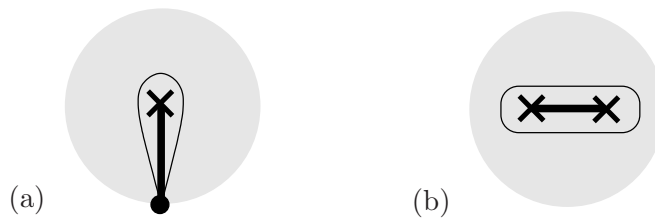,width=0.55\linewidth}
\caption{Deformations of thick curves into thin ones}
\label{deform}
\end{center}
\end{figure}

\begin{lemma}
\label{AB=BA}
The result of resolutions of several intersection points of a multicurve $C$ does not depend on the order of resolutions. In particular, if $p$ and $q$ are intersection points of $C$ then
$R_pR_q(C)=R_qR_p(C)$.

\end{lemma}

\begin{proof}
Each of $R_pR_q(C)$ and $R_qR_p(C)$ leads to a complete resolution of both intersection points
(note that in terms of the Remark~\ref{skein-th} a resolution of an intersection point is a local procedure).

\end{proof}

\begin{lemma}
For any multicurve $C$ there exists a sequence of resolutions turning $C$ into a sum of multicurves containing no intersections.

\end{lemma}

\begin{proof}
Let us count the total number of intersection points in $C$, counting an intersection of a thick curve with a thin with 
multiplicity $2$, an intersection of two thick curves in a regular point with multiplicity $4$, and an intersection of two thick curves in an orbifold point with multiplicity $2$. Then each resolution reduces the total number of intersections. 

\end{proof}

For a finite union of intersection points $P=\{p_1,\dots,p_k\}$ of $C$ we denote by $R_P(C)$ the resolution of all points of $P$.

\subsection{Laurent polynomials corresponding to curves on orbifolds}

In this section we define Laurent polynomial $x_\gamma$ (in a given arbitrary cluster) for each  curve $\gamma$ on orbifold.
Our definition comes from comparing the curve with its lift in some unfolding. To see the independence 
from the choice of unfolding we will use an interpretation of $x_\gamma$ (for a non-self-intersecting curve $\gamma$) as a lambda length of $\gamma$. 

\subsubsection{Lambda lengths of closed curves on surfaces}

Let us fix a triangulation $T$ on a surface $S$, and assign to the cluster variables $\{x_\alpha\}_{\alpha\in T}$ the values equal to the lambda lengths of the corresponding arcs. According to~\cite{FT}, given an arc $\gamma$ on $S$, the value of the cluster variable $x_\gamma$ also equals the lambda length of $\gamma$ (i.e. $x_\gamma$ represents lambda length of $\gamma$ as a function of lambda lengths of the arcs in the initial triangulation).

\begin{definition}
\label{lambdacl}
The  {\it lambda length} $\lambda_\gamma$ of a closed curve $\gamma$ is defined as
$\lambda_\gamma=e^{l(\gamma)/2}$, where $l(\gamma)$ is the hyperbolic length of the geodesic representative in the class of curves freely isotopic to $\gamma$.
\end{definition}

The Laurent polynomials associated to closed curves on surfaces are defined in~\cite[Definition~3.12]{MSW2} using band graphs. As proved in~\cite{MSW2}, these Laurent polynomials satisfy the skein relations. It is a widely known (but probably not written anywhere) fact that these Laurent polynomials are also 
equal to lambda lengths of closed curves. 

In this section we prove this fact for non-self-intersecting curves (i.e., for closed loops).

The following lemma is an easy exercise in hyperbolic geometry.

\begin{lemma}
\label{annulus}
Let $S$ be an annulus with one boundary marked point at each boundary component.
Let $\gamma$ be the closed loop on $S$. Then $x_\gamma$ represents the lambda length of $\gamma$.

\end{lemma}

\begin{lemma}
\label{length on surf}
Let $\gamma$ be a closed loop on an unpunctured surface $S$. Then  $x_\gamma$ represents the lambda length of $\gamma$.


\end{lemma}

\begin{proof}
First, suppose that the loop $\gamma$ is non-separating.
Let $q$ be a point on $\gamma$ and $p$ be a marked point in some boundary component of $S$ (it does exist since $S$ is unpunctured). Let $\alpha_1$ and $\alpha_2$ be two disjoint non-self-intersecting paths from $p$ to $q$ approaching $\gamma$ from two different sides and such that $\alpha_i\cap \gamma=q$ (here we use that $\gamma$ is non-separating). Consider the arcs $\gamma_i=\alpha_i\gamma\alpha_i^{-1}$, $i=1,2$
(see Fig.~\ref{a11emb}). Notice that these arcs cut off an annulus, and this annulus contains the curve $\gamma$.
Consider a triangulation of $S$ containing the arcs $\gamma_1$ and $\gamma_2$. 
In the cluster corresponding to this triangulation the function $x_\gamma$ is expressed through the lambda lengths of 
$\gamma_1$, $\gamma_2$ and two arcs lying in the annulus. Using Lemma~\ref{annulus} we see that  $x_\gamma$ is the lambda length of $\gamma$. 

\begin{figure}[!h]
\begin{center}
\psfrag{g}{$\gamma$}
\psfrag{g1}{$\gamma_1$}
\psfrag{a1}{$\alpha_1$}
\psfrag{g2}{$\gamma_2$}
\psfrag{a2}{$\alpha_2$}
\epsfig{file=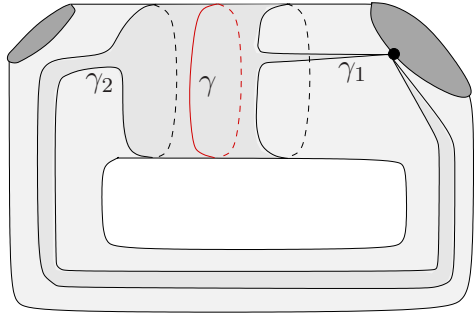,width=0.39\linewidth}
\caption{Construction of curves $\gamma_1$ and $\gamma_2$ for a non-separating closed loop $\gamma$}
\label{a11emb}
\end{center}
\end{figure}

Now, suppose that $\gamma$ is a separating curve. Let $S_1$ and $S_2$ be the connected components of  $S\setminus \gamma$. If each of $S_1$ and $S_2$ contains at least one marked point, we proceed as above (i.e. we build an annulus and then apply  Lemma~\ref{annulus}).
So, we may assume that $S_1$ contains no marked points, which implies that $S_1$ is a surface of genus at least $1$.
Let $\alpha$ be an arc which emanates from a marked point, then intersects $\gamma$, then goes around some non-separating non-self-intersecting loop inside $S_1$, then intersects $\gamma$ again and finally returns to the marked point (see Fig.~\ref{f-lambda}(a) and Fig.~\ref{f-lambda}(b)). Let $p_1$ and $p_2$ be the two points of intersection $\gamma\cap\alpha$.

Consider the resolution $R=R_{p_1,p_2}(\gamma\cup\alpha)$ of the multicurve $\gamma\cup \alpha$ in points $p_1$ and $p_2$.
By Lemma~\ref{skein-surface} we have
\begin{equation}
\label{1}
x_\gamma x_\alpha=x_R.
\end{equation}
Note that $R$ contains no separating curves and all curves of $R$ lie in the shaded area of the surface, so the value of $x_R$ is equal to the product of lambda lengths of the curves in $R$. Further, the value of $x_\alpha$ is also equal to the lambda length of $\alpha$. 
Thus, it is sufficient to prove that lambda lengths of curves included in the equation~(\ref{1}) satisfy similar relation
\begin{equation}
\label{2}
\lambda_\gamma \lambda_\alpha=\lambda_R,
\end{equation}
where $\lambda_\alpha$ and $\lambda_\alpha$ are the lambda lengths of $\alpha$ and $\gamma$, and $\lambda_R$ denoted the product of lambda lengths of the curves in $R$.

\begin{figure}[!h]
\begin{center}
\psfrag{a_}{\small (a)}
\psfrag{b_}{\small (b)}
\psfrag{c}{\small (c)}
\psfrag{a}{\scriptsize $\alpha$}
\psfrag{S1}{$S_0$}
\psfrag{g}{\scriptsize $\gamma$}
\psfrag{r1}{\scriptsize $\rho_1$}
\psfrag{r2}{\scriptsize $\rho_2$}
\psfrag{r}{\scriptsize $\rho$}
\psfrag{t}{\scriptsize $\tau$}
\epsfig{file=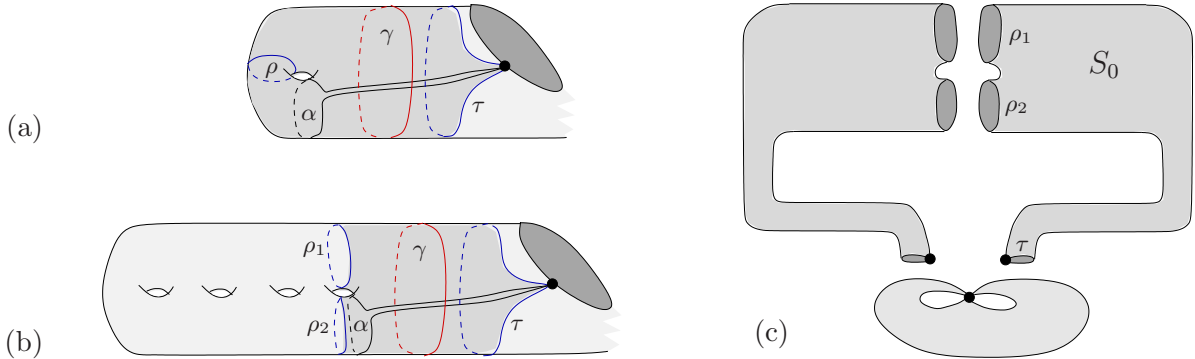,width=0.99\linewidth}
\caption{To the proof of Lemma~\ref{length on surf}: construction of $S_0$ and $S'$}
\label{f-lambda}
\end{center}
\end{figure}

To prove~(\ref{2}), we cut the shaded area $S_0$ out of $S$ (by an arc $\tau$ and either a geodesic loop $\rho$ in case of genus $1$
or by two geodesics loops $\rho_1$ and $\rho_2$ in case of higher genus, see Fig.~\ref{f-lambda}(a) and 
 Fig.~\ref{f-lambda}(b) respectively). Note that geodesic curves included in the relations~(\ref{1}) and~(\ref{2}) all lie inside the obtained area $S_0$ (since some representatives of the same isotopy class do lie there, and the boundary of $S_0$ is geodesic).
We will build a new surface $S'$ containing $S_0$ such that none of the curves participating in the equations
will be separating on $S'$. Then we know from above that the value of $x_\beta$ will be equal to $\lambda_\beta$ for all curves $\beta$ participating in~(\ref{1}) and~(\ref{2}), and so the formula~(\ref{2}) immediately follows from~(\ref{1}). 
Since the surface $S_0$ is embedded in $S'$ isometrically, the same formula for lambda lengths holds for curves on $S$, which implies that the value of $x_\gamma$ is equal to $\lambda_\gamma$ as required.

To build the surface $S'$ we proceed as on Fig.~\ref{f-lambda}(c): we take two copies of $S_0$, glue them together along the boundary components containing no marked points, then take a triangle with three vertices identified and attach free
boundary components to two of the boundary components of the triangle.
None of the curves contained in $S_0$ is separating in the constructed surface $S'$. 

\end{proof}

\subsubsection{Laurent polynomials for curves on orbifolds}
\label{Laurent poly}

Now we are ready to define the Laurent polynomials for curves on orbifolds in a given cluster.

For an arc or a pending arc $\gamma$ the Laurent polynomial $x_\gamma$ is the Laurent expansion of the cluster variable
corresponding to $\gamma$.

Let $S$ be a (partial) unfolding of $\h\O$. Let $\gamma$ be an arc or a pending arc.
Let $\bar \gamma$ be a connected component of the lift of $\gamma$ to $S$.
In view of Remark~\ref{spec}, $x_\gamma$ coincides with the specialization of the Laurent polynomial $x_{\bar \gamma}$.
This motivates the following definition.


\begin{definition}[Laurent polynomials for semi-closed loops]
\label{def-for semi}
Let $\gamma$ be a semi-closed loop on $\h\O$, let $S$ be an unfolding of $\h\O$.
Let $\bar \gamma$ be a connected component of the lift of $\gamma$ to $S$. 
Then define $x_\gamma=x_{\bar \gamma}|_{\h\O}$.

\end{definition}

\begin{remark}
The expression $x_{\bar \gamma}$ in Definition~\ref{def-for semi} depends on the choice of the unfolding $S$ of $\h\O$, 
however, the result $x_\gamma$ does not depend on the unfolding. Indeed, by Lemma~\ref{length on surf}, the value of $x_{\gamma}$ is the lambda length of $\bar \gamma$, which implies that after the specialization of variables the function
$x_\gamma$ is equal to $e^{l(\gamma)}$, where $l(\gamma)$ is the length of $\gamma$ (independently on the choice of the unfolding $S$). Therefore, the value of $x_\gamma$ depends on lambda lengths of arcs of a triangulation only, and thus $x_\gamma$ does not depend on the unfolding $S$.

By the same reason, the definition of $x_\gamma$ does not depend on a connected component $\bar \gamma$ in the lift of 
$\gamma$.

\end{remark}

The remark above motivates the following definition.

\begin{definition}
\label{lambda-semi}
The {\it lambda length} of a semi-closed loop $\gamma$ is defined by $\lambda_\gamma=e^{l(\gamma)}$, where $l(\gamma)$ is the length of $\gamma$.
\end{definition}

\begin{definition}[Laurent polynomials for closed curves]
\label{def-for loops}
Let $\gamma$ be a closed curve on $\h\O$, let $S$ be an unfolding of $\h\O$ such that $\gamma$ lifts well in $S$
(it does exist by Lemma~\ref{lifts well}).
Let $\bar \gamma$ be a connected component in the lift of $\gamma$ to $S$. 
Then define $x_\gamma=x_{\bar \gamma}|_{\h\O}$.

\end{definition}

\begin{remark}
 Similarly to the case of semi-closed loops, the expression $x_{\bar \gamma}$ in the definition of the Laurent polynomial $x_\gamma$ for a closed curve $\gamma$ depends on the choice of the unfolding. However, due to the fact that the Laurent polynomial $x_{\gamma_i}$ represents the lambda length of $\bar \gamma$ (see  Lemma~\ref{length on surf}), $x_\gamma$ represents the 
 lambda length of $\gamma$. Hence, the definition of $x_\gamma$ is independent on the choice of unfolding $S$, neither it depends on the choice of the connected component $\bar \gamma$ in the lift of $\gamma$.

\end{remark}

\begin{remark}
\label{specification}
Summarizing Definitions~\ref{def-for semi} and~\ref{def-for loops} we see that if a curve $\gamma\subset\h\O$ lifts well to an unfolding $S$ and $\bar \gamma$ 
is a connected component of the lift, then  $x_\gamma=x_{\bar \gamma}|_{\h\O}$.

This natural property will be heavily used below.

\end{remark}

\begin{definition}[Laurent polynomials for multicurves and self-intersecting curves]
\label{Lself}
Let $C=\gamma_1\cup \dots\cup \gamma_k$ be a multicurve. Define $x_C=x_{\gamma_1}\cdot\dots\cdot x_{\gamma_k}$.

For a formal sum $\Sigma=\sum\limits_{j=1}^m C_j$ of multicurves $C_j$ define $x_{\Sigma}=
\sum\limits_{j=1}^m x_{C_j}$.

Given a curve or a multicurve $C$, define $R(C)$ to be a complete resolution of all intersection points of $C$.
We consider $R(C)$ as a formal sum of multicurves, each summand having non-self-intersecting components.

For a self-intersecting curve $\gamma$ on $\h\O$ define the Laurent polynomial $x_\gamma$ as 
$x_\gamma=x_{R(\gamma)}$
where $R(\gamma)$ is the complete resolution of all intersection points of $\gamma$.

\end{definition}

\begin{definition}[Constants for contractible curves]
\label{const}
If $\gamma$ is a contractible curve then

- if $\gamma$ is a closed contractible curve then $x_\gamma=-2$;

- if $\gamma$ is a closed contractible curve around an orbifold point, then $x_\gamma=0$;

- if $\gamma$ is a contractible curve with both ends in an orbifold point then $x_\gamma=2$;

\end{definition}

\begin{remark} 
\label{rem loop var}
One can see that Definition~\ref{const} agrees with Definition~\ref{def skein}.

\end{remark}

\subsection{Proof of skein relations on orbifolds}
\label{skein-proof}
In this section we show the skein relations for the orbifold by proving the following theorem.

\begin{theorem}
\label{skein-list}
Let $C$ be a multicurve on an unpunctured orbifold $\h\O$. Let $p$ be an intersection point of $C$ (or a point of a self-intersection of some curve in $C$). Then $x_C=x_{R_p(C)}$, where $R_p(C)$ is the resolution of $C$ at the intersection point $p$
as defined in Table~\ref{skein-table}. 

\end{theorem}  

To prove Theorem~\ref{skein-list} we consider several cases. 
First, notice that if the point $p$ is a point of self-intersection
of one curve $\gamma_i\in C$, then there is nothing to prove (the statement holds by the definition of $x_{\gamma_i}$ for a self-intersecting curve $\gamma_i$). This holds for both regular and orbifold self-intersection points. 
So, we may assume that $p$ is an intersection point of two curves $\gamma_1$ and $\gamma_2$. Clearly, it is sufficient to prove the theorem for the case $C=\gamma_1\cup \gamma_2$. 
Now we are left with the following cases:

\begin{itemize}
\item[1.] $p$ is regular. Then either
\begin{itemize}
\item[1a.]
 both $\gamma_1$ and $\gamma_2$ lift well in some unfolding $S$ (in particular, this works if at least one of $\gamma_i$ is not a closed curve), this case is considered in Lemma~\ref{a}; 
\item[1b.] or both $\gamma_1$ are $\gamma_2$ are closed curves, this case is considered in Lemma~\ref{c}.
\end{itemize}

\item[2.] $p$ is an orbifold point, $\gamma_1$ are $\gamma_2$ are distinct thick curves. This case in considered in Lemma~\ref{orb,two}.

\end{itemize}

Hence, the proof of Theorem~\ref{skein-list} is reduced to Lemmas~\ref{a}--\ref{orb,two}.

\begin{lemma}
\label{a}
Let $\gamma_1$ and $\gamma_2$ be two curves intersecting at a regular point $p$. Assume also that there exists an unfolding of $\h\O$ such that each of the two curves $\gamma_1$ and $\gamma_2$ lifts well. Then  $x_{\gamma_1\cup \gamma_2}=x_{R_p(\gamma_1\cup\gamma_2)}$.

In particular, this equation holds if at least one of $\gamma_1$ and $\gamma_2$ is not a closed curve.

\end{lemma}

\begin{proof}
The statement follows from the definition of resolution $R_p(C_1\cup C_2)$.

More precisely, let $S$ be an unfolding of $\h\O$ such that each of the two curves $\gamma_1$ and $\gamma_2$ lifts well. Let $\bar p\in S$ be any lift of $p$, denote by $\bar \gamma_1\subset S$ and $\bar \gamma_2\subset S$ connected components of the lifts of $\gamma_1$ and $\gamma_2$ respectively containing $\bar p$. Then, according to Table~\ref{skein-table}, $R_{\bar p}(\bar \gamma_1\cup\bar \gamma_2)$ projects to $R_{p}(\gamma_1\cup\gamma_2)$ (see also Remark~\ref{factor of res}).
On the other hand, $x_{\bar \gamma_1}x_{\bar \gamma_2}=x_{R_{\bar p}(\bar \gamma_1\cup\bar \gamma_2)}$ as skein relations hold on the surface $S$. Since $x_\gamma=x_{\bar \gamma}|_{\h\O}$ for each curve $\gamma$ which lifts well, this implies  
$x_{\gamma_1\cup \gamma_2}=x_{\gamma_1}x_{\gamma_2}=x_{R_p(\gamma_1\cup\gamma_2)}$.

Now, if none of $\gamma_1$ and $\gamma_2$ is a closed curve, then both $\gamma_1$ and $\gamma_2$ lift well in any unfolding, and we may apply the reasoning above. If $\gamma_1$ is not a closed curve but $\gamma_2$ is a closed curve, then
by Lemma~\ref{lifts well} there exists an unfolding $S$ where $\gamma_2$ lifts well, so we can also apply the reasoning above.

\end{proof}

\begin{lemma}
\label{c}
Let $\gamma_1$ and $\gamma_2$ be two curves intersecting at a regular point $p$. 
Then  $x_{\gamma_1\cup \gamma_2}=x_{R_p(\gamma_1\cup\gamma_2)}$.

\end{lemma}

\begin{proof}
In view of Lemma~\ref{a} it is sufficient to prove the statement for the case when both $\gamma_1$ and $\gamma_2$
are closed curves. 

It is easy to see that for every non-contractible closed curve on $\h\O$ there exists an arc intersecting this curve. 
Let $\alpha$ be an arc intersecting $\gamma_2$ at some point $q$
($\alpha$ may have more intersection points with $\gamma_2$ and $\gamma_1$).
%
Since $\alpha$ is an arc we may apply Lemma~\ref{a} to resolve the intersection $q$
and get $x_{\gamma_2}x_{\alpha}=x_{R_q(\gamma_2\cup\alpha)}$, which implies
$$
x_{\gamma_1}x_{\gamma_2}x_{\alpha}=x_{\gamma_1}x_{R_q(\gamma_2\cup\alpha)}.
$$
Note that  as $\alpha$ is an arc and $\gamma_2$ is a closed curve, the resolution $R_q(\gamma_2\cup\alpha)$ is a sum of two ordinary curves, so we can apply  Lemma~\ref{a} to resolve the intersection $p$ and obtain
$$
x_{\gamma_1}x_{R_q(\gamma_2\cup\alpha)}=x_{R_p(\gamma_1\cup R_q(\gamma_2\cup\alpha))}.
$$
On the other hand, by  Lemma~\ref{AB=BA} we have $R_p(\gamma_1\cup R_q(\gamma_2\cup\alpha))=R_q(\alpha\cup R_p(\gamma_1\cup\gamma_2))$, which implies
 $x_{R_p(\gamma_1\cup R_q(\gamma_2\cup\alpha))}=x_{R_q(\alpha\cup R_p(\gamma_1\cup\gamma_2))}$.
Since $\alpha$ is not a closed curve, we apply Lemma~\ref{a} to have
$$
x_{R_q(\alpha\cup R_p(\gamma_1\cup\gamma_2))}=x_{\alpha}x_{R_p(\gamma_1\cup\gamma_2)}.
$$
Summarizing the above computation, we obtain
$$
x_{\gamma_1}x_{\gamma_2}x_{\alpha}=x_{\alpha}x_{R_p(\gamma_1\cup\gamma_2)}.
$$
Since the ring of Laurent polynomials has no zero divisors, this implies that 
$x_{\gamma_1}x_{\gamma_2}=x_{R_p(\gamma_1\cup\gamma_2)}$ which proves the lemma.

\end{proof}

\begin{lemma}
\label{orb,two}
Let $\gamma_1$ and $\gamma_2$ be two thick curves incident to the same  orbifold point $p$. 
Then  $x_{\gamma_1\cup \gamma_2}=x_{R_p(\gamma_1\cup\gamma_2)}$.

\end{lemma}

\begin{proof}
There are five possibilities for a pair of thick curves shown in Fig.~\ref{orb2} (depending on the types of thick curves $\gamma_1$ and $\gamma_2$ and the number of common vertices). 
In all of these cases  the resolution $R_p(\gamma_1\cup\gamma_2)$ contains no closed curves, which implies that
all curves involved in  $R_p(\gamma_1\cup\gamma_2)$ as well as the curves $\gamma_1$ and $\gamma_2$ lift well in any unfolding $S$.  So, the required equation follows from the surface version of skein relations.

\end{proof}

\begin{figure}[!h]
\begin{center}
\psfrag{g1}{\scriptsize $\gamma_1$}
\psfrag{g2}{\scriptsize $\gamma_2$}
\psfrag{p}{\scriptsize $p$}
\epsfig{file=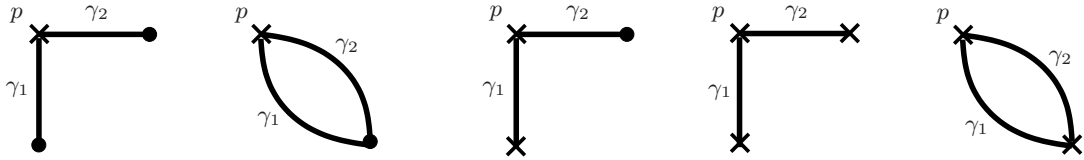,width=0.9\linewidth}
\caption{Pairs of thick curves intersecting at an orbifold point}
\label{orb2}
\end{center}
\end{figure}

In the case of principal coefficients, making use of Theorem~\ref{skein-list},~\cite[Propositions~6.4--6.6]{MSW2}, and then specializing variables, we obtain the following relations. 

\begin{lemma}
\label{skein-principal}
$$x_C=Y_+x_{C_+}+Y_-x_{C_-}+Y_=x_{C_=},$$
where $Y_+$, $Y_-$, $Y_=$ are monomials in variables $y_i$ computed in the following way:
$$Y_s=\prod\limits_{i=1}^n y_i^{p_{i,s}},\qquad s\in\{+,-,=\},$$
where $p_{i,s}$ is equal to one half of the difference of the intersection numbers of the elementary lamination $L_i$ with $C$ and $C_s$.

\end{lemma} 

The intersection numbers of laminations with multicurves on orbifolds should be redefined as follows: every intersection of a lamination with a thick curve counts twice.

\section{Bases $\B^\circ$, $B^\s$ and $\B$ on orbifolds}
\label{bases}
\subsection{Definitions}

Recall from~\cite{MSW2} that given a closed loop $\gamma$, the {\it bangle} $Bang_k(\gamma)$ is a union of $k$ loops isotopic to $\gamma$, and the {\it bracelet} $Brac_k(\gamma)$ is a closed curve obtained by concatenating $\gamma$
exactly $k$ times (see Fig.~\ref{b-b-closed}). A {\it band} $Band_k(\gamma)$ is defined in~\cite{T2} as an average of all possible end pairings of $k$ copies of $\gamma\setminus I$, where $I\subset\gamma$ is a short interval. In other words, $Band_k(\gamma)$ can be considered as a weighted sum of unions of bangles and bracelets,
$$
Band_k(\gamma)=\frac{1}{k!}\sum\limits_{(k_1\le k_2\le\dots\le k_m)\in P(k)}(k_1,\dots,k_m)\,Brac_{k_1}(\gamma)\cup\dots\cup Brac_{k_m}(\gamma),
$$
where $P(k)$ is the set of all partitions of $k$, and $(k_1,\dots,k_m)$ is the number of permutations in the symmetric group $S_k$ with given cyclic structure.

\begin{figure}[!h]
\begin{center}
\epsfig{file=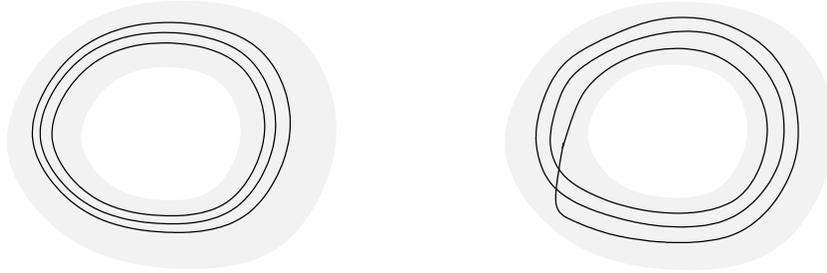,width=0.69\linewidth}
\caption{A bangle $Bang_3(\gamma)$ and a bracelet $Brac_3(\gamma)$ for a closed loop}
\label{b-b-closed}
\end{center}
\end{figure}

We define similar curves for a semi-closed curve $\gamma$ as follows: if we denote by $\bar \gamma$ the closed curve which is a lift of $\gamma$ in some degree two unfolding, then the lifts of $Bang_k(\gamma)$ and $Brac_k(\gamma)$ should coincide with $Bang_k(\bar\gamma)$ and $Brac_k(\bar\gamma)$. Once bracelets are defined, we can define $Band_k(\gamma)$ via the formula above.
 
More precisely, a {\it bangle} $Bang_k(\gamma)$ is a union of $k$ semi-closed loops isotopic to $\gamma$ (see Fig.~\ref{b-b-semi-closed}(a)).
A {\it bracelet} $Brac_k(\gamma)$ is defined differently for odd and even $k$. Denote by $q_1$ and $q_2$ the endpoints of $\gamma$. Then $Brac_{2m+1}(\gamma)$ 
is a semi-closed curve with endpoints $q_1$ and $q_2$ going around $q_1$ and $q_2$ exactly $m$ times (see Fig.~\ref{b-b-semi-closed}(b)). Finally, $Brac_{2m+2}(\gamma)$ is a semi-closed curve with both endpoints being $q_1$ (or $q_2$) going around $q_1$ and $q_2$ exactly $m$ times (see Fig.~\ref{b-b-semi-closed}(c)).

Note, that the equality $Bang_1(\gamma)=Brac_1(\gamma)=Band_1(\gamma)=\gamma$ still holds for semi-closed loops (since a closed loop $\gamma'$ around $\gamma$ is isotopic to $\gamma$).

\begin{figure}[!h]
\begin{center}
\psfrag{a}{\small (a)}
\epsfig{file=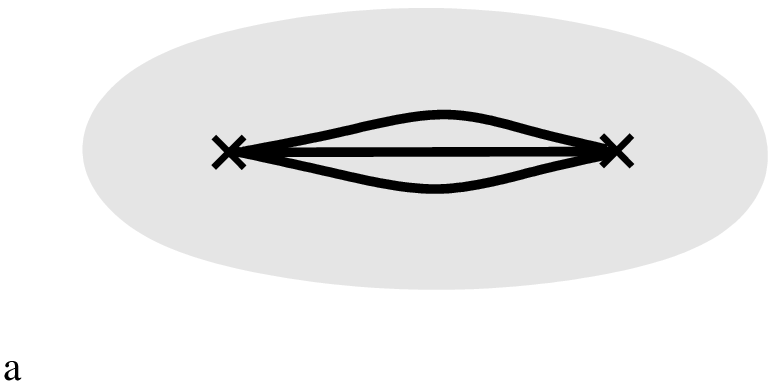,width=0.3\linewidth}\qquad\quad
\psfrag{c}{\small (b)}
\epsfig{file=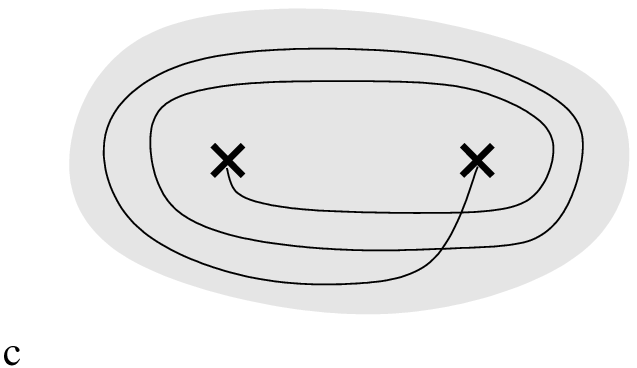,width=0.27\linewidth}\qquad\quad
\psfrag{b}{\small (c)}
\epsfig{file=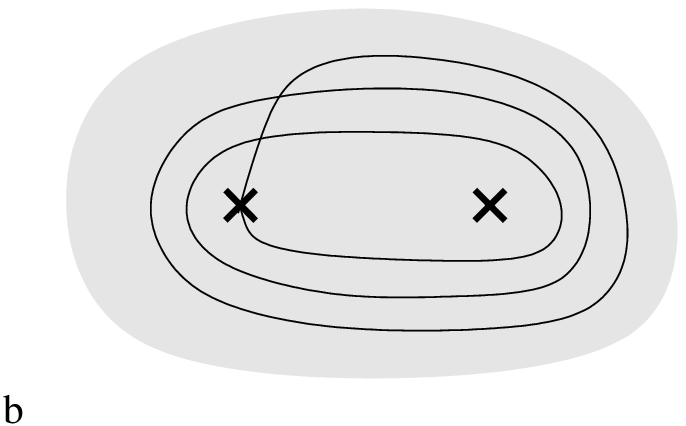,width=0.25\linewidth}\qquad
\caption{(a) A bangle $Bang_3(\gamma)$; (b) a bracelet $Brac_5(\gamma)$; (c) a bracelet $Brac_6(\gamma)$ for a semi-closed loop $\gamma$}
\label{b-b-semi-closed}
\end{center}
\end{figure}

The Laurent polynomials associated to bangles and bracelets can be computed via Definition~\ref{Lself}. One can easily check that the obtained formulae coincide with ones for surfaces (see~\cite{MSW2,T2}): if $\gamma$ is a semi-closed curve, then $x_{Bang_k(\gamma)}=(x_\gamma)^k$, $x_{Brac_k(\gamma)}=T_k(x_\gamma)$ and $x_{Band_k(\gamma)}=U_k(x_\gamma)$, where $T_k$ is the Chebyshev polynomial of the first kind defined by initial conditions $T_0(x)=2$, $T_1(x)=x$ and recurrent relation $T_k(x)=xT_{k-1}(x)-T_{k-2}(x)$, and $U_k$ is the Chebyshev polynomial of the second kind defined by initial conditions $U_0(x)=1$, $U_1(x)=x$ and recurrent relation $U_k(x)=xU_{k-1}(x)-U_{k-2}(x)$. 
\medskip

Following~\cite{MSW2}, we define $\mathcal C^\circ$- and $\mathcal C$-compatibility.

\begin{definition}
A finite collection $C$ of arcs, closed loops, pending arcs and semi-closed loops on $\h\O$  is $\mathcal C^\circ$-{\it compatible}
if no two elements of $C$ cross each other. The set of all  $\mathcal C^\circ$-compatible collections
in $\h\O$ is denoted by  $\mathcal C^\circ(\h\O)$.

A finite collection $C$ of arcs, pending arcs and bracelets is $\mathcal C$-{\it compatible} if
\begin{itemize}
\item no two curves intersect each other except for self-intersection of bracelets;

\item given a closed loop or a semi-closed loop $\gamma$, there is at most one $k\ge 1$ such that $k$-th bracelet
$Brac_k(\gamma)$ lies in $C$. Moreover, there is at most one copy of this bracelet  $Brac_k(\gamma)$ in $C$.

\end{itemize}

 The set of all  $\mathcal C$-compatible collections
in $\h\O$ is denoted by  $\mathcal C(\h\O)$.


\end{definition}

After we extended the definition of   $\mathcal C^\circ$- and $\mathcal C$-compatibility to the orbifold case, the definition of the bases 
$\B^\circ$ and $\B$ coincides with the one given in~\cite{MSW2} for the surface case:

\begin{definition}
Given a curve $\gamma$, let $x_\gamma$ be a Laurent polynomial defined in Section~\ref{Laurent poly}.
Then
$$
\B^\circ=\left\{\prod\limits_{\gamma\in C} x_\gamma \ | \ C\in \mathcal C^\circ(\h\O)  \right\}
$$
and 
$$
\B=\left\{\prod\limits_{\gamma\in C} x_\gamma \ | \ C\in \mathcal C(\h\O)  \right\}.
$$
$\B^\circ$ is called the {\it bangle basis}, and $\B$ is called the {\it bracelet basis}. 

Substituting bracelets by bands in the definition of $\mathcal C$-compatibility, we obtain the notion of $\mathcal C^\s$-compatibility and the set of $\mathcal C^\s(\h\O)$ of all  $\mathcal C^\s$-compatible collections. The {\it band basis} is then defined as 
$$
\B^{\s}=\left\{\prod\limits_{\gamma\in C} x_\gamma \ | \ C\in \mathcal C^{\s}(\h\O)  \right\}.
$$
\end{definition}

\subsection{Relations between $\B^\circ$, $\B^\s$ and $\B$}
\label{rels}

Similarly to the surface case, each element of the bracelet basis $\B$ is an integer  
linear combination of elements of $\B^\circ$.
More precisely, the only type of elements of $\B$ not contained in $\B^\circ$ is a bracelet and
$$x_{Brac_k(\gamma)}=T_k(x_\gamma),$$
where $T_k$ is a Chebyshev polynomial of the first kind (see above).
Note that $T_k(x)=x^k+\delta$, where $\delta$ is a sum of smaller powers of $x$,
which implies that $x^k$ is an integer linear combination of $T_i(x)$, $i\le k$.
In other words, each element of $\B^\circ$ is an integer linear combination of elements of $\B$.

Exactly the same relation holds between $\B^\circ$ and $\B^\s$: each element of the band basis $\B^\s$ is an integer linear combination of elements of $\B^\circ$, and each element of $\B^\circ$ is an integer linear combination of elements of $\B^\s$. The reason is exactly the same: a Chebyshev polynomial $U_k$ of the second kind has the form $U_k(x)=x^k+\delta'$, where $\delta'$ is a sum of smaller powers of $x$.  

We will prove that $\B^\circ$ is a basis for the cluster algebra $\A$, i.e. we will prove that

- elements of $\B^\circ$ belong to the cluster algebra (Lemma~\ref{in});

- elements of $\B^\circ$ span the cluster algebra (Lemma~\ref{span}); 

- elements of $\B^\circ$ are linearly independent (Theorem~\ref{independence}).

Then all the statements for the sets $\B$ and $\B^\s$ follow immediately from the fact that elements of $\B$ and $\B^\s$ are related to elements of $\B^\circ$ by a unitriangular integer linear transformation.

\section{Skein relations and elements of $\B^\circ$}
\label{reps}
In this section, we show that elements of $\B^\circ$ 
belong to cluster algebra and span it. We remind the reader that we consider cluster algebras originating from unpunctured orbifolds with at least two marked points at the boundary. 

\begin{lemma}
\label{in}
Elements of $\B^\circ$ 
belong to the cluster algebra.
\end{lemma}

\begin{proof}
%
We need to consider closed and semi-closed loops only. For closed loops, we apply~\cite[Proposition~4.5]{MSW2} without any changes. For semi-closed loops we apply the same method as  in the proof of~\cite[Proposition~4.5]{MSW2},
but we use skein relations for the configuration of curves shown in   Fig.~\ref{semi-closed} on the left
for the case when at least one boundary component of $\h\O$ contains two or more marked points, and the configuration of curves shown in 
 Fig.~\ref{semi-closed} on the right otherwise (one can see that  Fig.~\ref{semi-closed} is a counterpart
of~\cite[Figure 11]{MSW2} and~\cite[Figure 13]{MSW2}). After the resolution of all intersections, one of the summands will be a product of the Laurent polynomial associated to a closed loop around two orbifold points and boundary segments. Note that the former is equal to the Laurent polynomial associated to the semi-closed loop (as they have the same geodesic representative, cf. Remark~\ref{skein-th} and Fig.~\ref{deform}(b)).

\end{proof}

\begin{figure}[!h]
\begin{center}
\epsfig{file=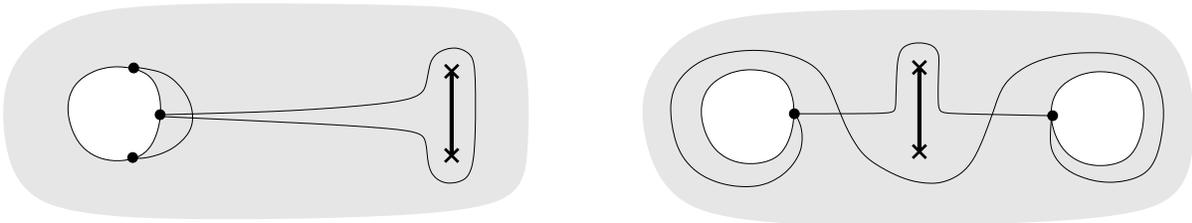,width=0.99\linewidth}
\caption{A multicurve whose resolution produces a product of a semi-closed loop and boundary segments}
\label{semi-closed}
\end{center}
\end{figure}

\begin{lemma}
\label{span}
$\B^\circ$ is a spanning set for the cluster algebra.

\end{lemma}

\begin{proof}
We need to show that every product of cluster variables can be represented as a sum of elements of $\B^\circ$.

Suppose that we have a collection of arcs and pending arcs $C$. 
Consider a complete resolution $R(C)$ of all intersection points of $C$. We get a formal sum of multicurves, each consisting of mutually non-intersecting
non-self-intersecting curves (i.e. we get a formal sum of collections of non-intersecting arcs, pending arcs, closed loops and semi-closed loops, and thus a formal sum of $\mathcal C^\circ$-compatible sets). 
Hence, $x_C=x_{R(C)}$ is expressed as a sum of elements of $\B^\circ$. 

\end{proof}

\section{Linear independence of $\B^\circ$}
\label{ind}

In this section, we show that the set $\B^\circ$ is linearly independent. Our proof follows the plan of the proof from~\cite{MSW2}. First, we show that a counterpart of~\cite[Theorem~5.1]{MSW2} holds (Lemma~\ref{5.1}), so we can make use of the notion of $\g$-vectors of elements of $\B^\circ$. Then we use tropical duality~\cite{NZ} and results of~\cite{MSW2} to associate $\g$-vectors of $\B^\circ$ to certain laminations on the ``reversed'' associated orbifold $\h\O^*$ (see Definition~\ref{reversed}), which, in view of the results of~\cite{FeSTu3}, implies bijection between $\g$-vectors of $\B^\circ$  and elements of $\Z^n$ (Theorem~\ref{g-bijection}). The application of~\cite[Proposition~2.13]{MSW2} will complete the proof.   

\begin{lemma}[cf.~\cite{MSW2}, Theorem 5.1]
\label{5.1}
Any element of $\B^\circ$ contains a unique term $\x^{\g}$ not divisible by
any coefficient variable, and the exponent vector of each other term is obtained from
$\g$ by adding a non-negative linear combination of columns of the extended exchange matrix $\t B$. 

\end{lemma} 

\begin{proof}
The proof follows from~\cite[Theorem~5.1]{MSW2}. Let $S$ be an unfolding of $\h\O$. As we have already mentioned, every element of $\B^\circ$ can be obtained from a corresponding element of $\B^\circ(S)$ by a specialization of variables, where $\B^\circ(S)$ is the bangle basis for the surface cluster algebra $\mathcal A(S)$. This immediately implies the existence of the leading term (i.e., the term $\x^{\g}$) and defines $\g$-vectors for all elements of $\B^\circ$. The second statement of the lemma follows from the surface version and the definition of the unfolding of exchange matrix (see e.g.~\cite{FeSTu2}).

\end{proof}

\begin{definition}
A {\it $\g$-vector} of an element $x_C$ of $\B^\circ$ is the multidegree of its leading term $\x^\g$.  
\end{definition}

\subsection*{Laminations and $\g$-vectors}
\label{g-lam}

First, we introduce {\it reversed} associated orbifold $\h\O^*$ and its triangulation $T^*$.

\begin{definition}
\label{reversed}
Let $\h\O$ be an associated orbifold with triangulation $T$. The {\it reversed} associated orbifold $\h\O^*$ is obtained from $\h\O$ in the following way: replace all orbifold points by special marked points, and all the special marked points by orbifold ones. To obtain the corresponding triangulation $T^*$ from $T$ all pending arcs should be replaced by double ones, and the double arcs should be replaced by pending ones. 

\end{definition}

\begin{remark}
\label{reversed-transposed}
If a triangulation $T$ on an associated orbifold  $\h\O$ is defined by a skew-symmetrizable matrix $B$, then the triangulation $T^*$ on $\h\O^*$ is defined by $-B^T$.    

\end{remark}

We will also need a notion of {\it reversed elementary lamination} $L_i^*$ on reversed associated orbifold $\h\O^*$.

\begin{definition}
\label{reversed-lam}
Given a triangulation $T$ of $\h\O$, a reversed elementary lamination $L_i^*$ is a lamination on $\h\O^*$ with shear coordinates $b_{T^*}(L_i^*)=(0,\dots,0,-1,0,\dots,0)$, where $-1$ is located on $i$-th place. Geometrically, $L_i^*$ is a ``reflection'' of the elementary lamination $L_i$ (w.r.t. $T*$) in the $i$th arc of $T*$. 

\end{definition}

According to~\cite[(1.13)]{NZ}, 
there is a duality between $\cc$-vectors and $\g$-vectors of cluster algebras which can be expressed in the following terms:
$$G_t^{B;t_0}=(C_{t_0}^{B_t^T;t})^T,$$
where $G_t^{B;t_0}$ is the matrix composed of $\g$-vectors of a seed $(t,B_t)$ in the initial seed $(t_0,B)$, and $C_{t_0}^{B_t^T;t}$ is the matrix composed of $\cc$-vectors of a cluster $t_0$ in the initial cluster $(t,B_t^T)$ (see Section~\ref{principal} for definitions). This duality holds in the assumptions of sign-coherence of $\cc$-vectors, which is true for cluster algebras from orbifolds~\cite[Theorem~14.1]{FeSTu3} (sign-coherence of $\cc$-vectors in full generality was recently proved in~\cite{GHKK}).

Since the rows of $C$-matrix are shear coordinates of elementary laminations (see~\cite[Theorem~9.1]{FeSTu3}), and the (negative) transposed matrix $B^T$ corresponds to the reversed associated orbifold (see Remark~\ref{reversed-transposed}), the duality can be reformulated in the following way.

\begin{lemma}
\label{g-arcs}
Let $T_o$ be a triangulation of an associated orbifold $\h\O$ corresponding to the initial cluster, and let $x_{\gamma}$ be an arbitrary cluster variable (i.e., $\gamma$ is some  arc on $\h\O$). Let $T$ be any triangulation of $\h\O$ containing $\gamma$, and let $L_{\gamma}$ be the corresponding elementary lamination with respect to $T$. Then 
$$\g(x_{\gamma})=-b_{T_0^*}(L_{\gamma}^*)$$  

\end{lemma}     

\begin{remark}
Note that Lemma~\ref{g-arcs} for arcs on unpunctured surfaces can also be deduced from the formula for $\g$-vectors given in~\cite[Corollary~6.15(1)]{MSW2}. 
\end{remark}

Comparing Lemma~\ref{g-arcs} with~\cite[Corollary~6.15]{MSW2}, we immediately obtain similar expression for closed loops on unpunctured surfaces.

\begin{lemma}
\label{g-loops-s}
Let $S$ be an unpunctured marked surface, and let $T_0$ be a triangulation of $S$ corresponding to the initial cluster. Let $x_{\gamma}$ be the element corresponding to a closed loop $\gamma$ on $S$. Then  $$\g(x_{\gamma})=-b_{T_0^*}({\gamma})$$

\end{lemma}

Here the closed loop $\gamma$ is understood as a lamination consisting of a single curve. 

\begin{remark}
\label{g-loops-band}
There is another way to prove Lemma~\ref{g-loops-s} based on investigation of the band graph of closed loop (see~\cite{MW,MSW2}). More precisely, given a loop $\gamma$ on $S$, one can find an arc $\t\gamma$ on $S$ which is very close to $\gamma$, see Fig.~\ref{g-loops-f}. It is easy to see that the vector $b_{T_0}({\gamma})$ of shear coordinates of $\gamma$ can be obtained from the vector $b_{T_0}(L_{\t\gamma}^*)$ of shear coordinates of the reversed elementary lamination of $\t\gamma$ by adding the vector $e_{\gamma_i}+e_{\gamma_j}-e_{\gamma_k}$. On the other hand, comparing the band graph of $\gamma$ and the snake graph of $\t\gamma$, one can easily see that the monomials without coefficients (which correspond to the leading terms, and thus to $\g$-vectors) differ exactly by $x_{\gamma_k}/x_{\gamma_i}x_{\gamma_j}$.

Notice that Fig.~\ref{g-loops-f} represents an easy case when the curve $\t\gamma$ intersects two
arcs only (namely, $\gamma_i$ and $\gamma_k$) incident to its basepoint. In more general setting one uses   
the notion of fan introduced in~\cite[Section~6.1]{MSW2}; this keeps the situation as simple as in the initial case and leads to exactly the same result.

The approach above seems to be suitable for generalizations to punctured case. For example, it allows an immediate generalization to punctured surfaces in the case of $T_0$ coming from ideal triangulation and having a conjugate pair in every puncture.       
\end{remark}

\begin{figure}[!h]
\begin{center}
\psfrag{g}{\scriptsize $\gamma$}
\psfrag{g1}{\scriptsize $\t\gamma$}
\psfrag{g2}{\scriptsize $\gamma_i$}
\psfrag{g3}{\scriptsize $\gamma_k$}
\psfrag{g4}{\scriptsize $\gamma_j$}
\epsfig{file=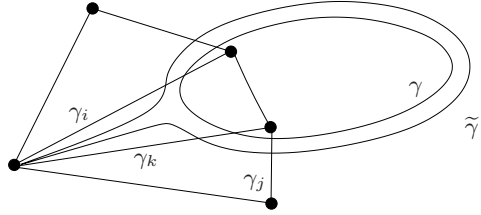,width=0.39\linewidth}
\caption{The arc $\t\gamma$ is very close to the closed loop $\gamma$ (with respect to the triangulation $T_0$)}
\label{g-loops-f}
\end{center}
\end{figure}

Since both $\g$-vectors and $\cc$-vectors on orbifolds can be obtained via specialization of initial variables, and $\cc$-vectors are exactly shear coordinates of elementary laminations, Lemma~\ref{g-loops-s} gives rise to a similar statement for closed loops on associated orbifolds without punctures (and special marked points). 

\begin{lemma}
\label{g-loops}
Let $\h\O$ be an associated orbifold without punctures and special marked points, and let $T_0$ be a triangulation of $\h\O$ corresponding to the initial cluster. Let $x_{\gamma}$ be the element corresponding to a closed loop $\gamma$ on $\h\O$. Then  $$\g(x_{\gamma})=-b_{T_0^*}({\gamma})$$

\end{lemma}


\begin{lemma}
\label{g-surjection}
The map $\g:\B^\circ\to\Z^n$ assigning to an element of $\B^\circ$ its $\g$-vector is surjective.

\end{lemma}

\begin{proof}
Given  any integer $n$-vector $b=(b_1,\dots,b_n)$, we find an element of $\B^\circ$ with $\g$-vector $b$.

Choose a triangulation $T_0$ of $\h\O$.  As it is shown in~\cite{FeSTu3}, there is a unique lamination $L'$ on $\h\O^*$ with shear coordinates $-b$ (with respect to $T_0^*)$. As it is easy to see, every single curve of any lamination is either closed (or semi-closed) loop, or a reversed elementary lamination $L_{\gamma}^*$ for some arc $\gamma$ on $\h\O$ (to obtain such $\gamma$ just shift every boundary end of every curve clockwise till the closest boundary marked point).

Now construct a non-intersecting collection $C$ of curves on $\h\O$. It will consist of closed loops of $L'$ (with semi-closed loops substituting loops around two special marked points in $L'$), and all those arcs $\gamma$ whose reversed elementary lamination $L_{\gamma}^*$ is contained in $L'$. By Lemmas~\ref{g-arcs} and~\ref{g-loops}, the element of $\B^\circ$ corresponding to the union of curves of $C$ has $\g$-vector $b$. 

\end{proof}

\begin{theorem}   
\label{g-bijection}
The map $\g:\B^\circ\to\Z^n$ assigning to an element of $\B^\circ$ its $\g$-vector is bijective. 

\end{theorem}

\begin{proof}
Consider  $\g:\B^\circ\to\Z^n$ first. We need to show that an element of $\B^\circ$ with $\g$-vector $b$ is unique, the surjectivity follows from Lemma~\ref{g-surjection}. Given an element $x_C$ of $\B^\circ$, we can define a lamination $\t L_C'$ on $\h\O^*$ in the following way. We keep all the closed (or semi-closed) curves of $C$ (with semi-closed curves substituted by loops around two special marked points), and take a reversed elementary lamination $L_{\gamma}^*$ for every $\gamma\in C$.

The considerations above show that the vector of shear coordinates of the lamination $L_C'$ with respect to any triangulation $T_0^*$ is equal to negative $g$-vector of $x_C\in\B^\circ$ with respect to $T_0$. Moreover, two distinct elements of $\B^\circ$ lead to two different laminations $L_C'$. According to~\cite[Lemma~6.6]{FeSTu3}, distinct laminations have distinct shear coordinates, which implies that two distinct elements of $\B^\circ$ have distinct $\g$-vectors.

\end{proof}

\begin{theorem}
\label{independence}
The set $\B^\circ$ is linearly independent.

\end{theorem}

The proof follows {\it verbatim} the proof of~\cite[Corollary~6.14]{MSW2} after a substitution of~\cite[Theorems~5.1 and 6.13]{MSW2} by their orbifold counterparts,  Lemma~\ref{5.1} and Theorem~\ref{g-bijection}.  In view of Section~\ref{rels}, this completes the proof of Theorem~\ref{bases-thm}.

\section{Positivity of $\B$}
\label{pos}

\begin{definition}
An additive basis $\mathbf b=\{x_i\}$ of a $\Z$-algebra is called {\it positive} if it has positive {\it structure constants},
i.e. for any $x_j,x_k\in \mathbf b$ one has $x_jx_k=\sum_i n_{jk}x_i$ with $n_{jk}\ge 0$.  
\end{definition}

It was conjectured in~\cite{FG0} that both bases $\B^\circ$ and $\B$ for cluster algebras from surfaces are positive.
However, it was demonstrated in~\cite[Exercise~20.5, Lecture~18]{T}
that for the bangle basis $\B^\circ$ the statement fails already in an annulus.
More precisely, let $p$ and $q$ be marked points in different boundary components of the annulus and let $\gamma_k$ be an arc connecting $p$ to $q$ and wrapping $k$ times, then  it is easy to check that positivity fails for
$\gamma_0\gamma_5$. 

In~\cite{T2} D.~Thurston proved positivity for the bracelet basis of the skein algebra of a surface (with punctures), which, in particular, implies positivity of the bracelet basis $\B$ on an unpunctured surface. As a corollary, we get similar statement for an unpunctured orbifold.

\begin{theorem}
\label{positive}
The bracelet basis $\B$ is positive.

\end{theorem}

The idea of proof in~\cite{T2} is to show that one can always resolve the crossings of a multicurve in such an order that negative terms (i.e. contractible loops) never arise.

We proceed as follows:
\begin{itemize}
\item[1)] use small deformations as in Fig.~\ref{deform} to turn all thick curves into thin ones;
\item[2)] substitute all orbifold points by punctures, so that we obtain a multicurve in a surface with punctures;
\item[3)] resolve the intersections in such an order that no contractible loops will be obtained at any step;
\item[4)] put the orbifold points back on their places;
\item[5)] if needed, deform the closed loops around two orbifold points to semi-closed curves and arcs around one orbifold point to 
pending arcs.

\end{itemize}

The closed contractible curves will never arise, so that we obtain non-negative coefficients at each summand (see Remark~\ref{noneg}).

\begin{remark}
\label{noneg}
In the last line of Table~\ref{skein-table} defining the resolution one can find
 $R(C)=C_++C_--2$. This does not prevent the bracelet basis from being positive. Indeed, the multicurve $C$ in this case contains a non-contractible thick curve (call it $\gamma$) with two ends at the same orbifold point $p$. Denote by $\gamma_1$ and $\gamma_2$ closed loops homotopic to the curves $\gamma_+$ and $\gamma_-$ (see Table~\ref{skein-table}). Then the skein relation becomes
$$R_p(\gamma)=Brac_2(\gamma_1)+Brac_2(\gamma_2)+2,$$
so it also has non-negative coefficients. 

\end{remark}

\section{Atomic bases for $C_n^{(1)}$ and finite type}
\label{atomic}

\begin{definition}
An element of a cluster algebra $\A$ is {\it positive} if its Laurent expansion in every cluster has non-negative coefficients only, denote the set of positive elements by $\A^+$. An additive basis $\mathbf b=\{x_i\}$ is {\it atomic} if $z\in\A^+$ if and only if $z$ can be written as a linear combination of $\{x_i\}$ with non-negative coefficients.  

\end{definition}

Cerulli Irelli~\cite{C} proved that cluster monomials form an atomic basis of skew-symmetric cluster algebras of finite type. In~\cite{DT}, Dupont and Thomas showed that the basis of a cluster algebra $\t A_{p,q}$ constructed by Dupont in~\cite{D} is atomic. They also gave a similar proof for the $A_n$ case. Their atomic basis for $\t A_{p,q}$ coincides with the bracelet basis ($\t A_{p,q}$ is represented by an annulus with $p$ and $q$ marked points on the boundaries). The cluster monomial basis of $A_n$ is also a particular case of the bracelet basis. In~\cite{GM}, Gunawan and Musiker give a combinatorial proof of the fact the cluster monomial basis of the algebra of type $D_n$ is atomic.

Note that cluster algebras of classical finite and affine types originate from surfaces and orbifolds: $A_{n}$ can be realized as a disc with $n+3$ boundary marked points, $B_n$ as a disc with one special marked point and $n+1$ boundary marked points, $C_n$ as a disc with one orbifold point and $n+1$ boundary marked points, $D_{n}$ as a once punctured disc with $n$ boundary marked points, $\t A_{p,q}$ as an annulus with $p$ and $q$ marked points on the boundary components, and $C_n^{(1)}$ as a disc with two orbifold points and $n$ boundary marked points. In particular, $A_{2n-1}$ is an unfolding of $C_n$, $D_{n+1}$ is an unfolding of $B_n$, and $\t A_{n,n}$ is an unfolding of $C_n^{(1)}$. The aim of this section is to use unfoldings to adapt the proof from~\cite{DT} to the orbifold case. Furthermore, it was shown in~\cite{FeSTu2} that a cluster algebra of type $F_4$ also can be unfolded to a cluster algebra of type $E_7$. As an application of our methods, we extend the result of~\cite{C} to the skew-symmetrizable case.   

It was shown in~\cite{MSW,MSW2} that cluster variables $x_{\gamma}$ and bracelets $Brac_k(\gamma)$ on an unpunctured surface are positive. Using unfolding, we can see that cluster variables and bracelets on an unpunctured orbifold are also positive (cluster variables on any orbifold are positive by~\cite[Theorem 13.1]{FeSTu3}). This immediately implies that all elements of $\B$ are positive, and thus any linear combination of elements of $\B$ with non-negative coefficients is positive.

The goal of this section is to prove the following theorems.

\begin{theorem}
\label{atomF}
Cluster monomial bases of skew-symmetrizable cluster algebras of finite type are atomic.

\end{theorem}

\begin{theorem}
\label{atom}
The bracelet basis on the disc with two orbifold points and $n$ marked points is the atomic basis of the cluster algebra of affine type $C_n^{(1)}$.

\end{theorem}    

First, let $\B$ be the cluster monomial basis of a skew-symmetric cluster algebra $\A$ of finite type. Let $\x$ be any cluster. Take any cluster variable  $x\notin \x$, and take any cluster monomial $x_{C_0}$ containing $x$ as a factor. Consider the Laurent expansion of $x_{C_0}$ in the cluster $\x$ (we will call it {\em $\x$-expansion} of $x_{C_0}$). 
Following~\cite{CL}, we say that a term of the $\x$-expansion of $x_{C_i}$ is a {\em proper Laurent monomial} if it has negative degree with respect to at least one variable, and the cluster algebra $\A$ has {\em the proper Laurent monomial property} if for any two clusters $\x$ and $\x'$ of $\A$, every monomial in $\x'$ in which at least one factor does not belong to $\x$ is a linear combination of proper Laurent monomials in $\x$.

\begin{lemma}[\cite{CKLP}, Corollary 3.4]
\label{plmp}
Any skew-symmetric cluster algebra has the proper Laurent monomial property.

\end{lemma}

Now consider any element $z=\sum_i\lambda_ix_{C_i}$, where $x_{C_i}\in\B$ are cluster monomials. Choose any $C\in\{C_i\}$. Let $\x$ be a cluster containing all the cluster variables that are factors of $x_{C}$. Clearly, the $\x$-expansion of $x_{C}$ is just a monomial. Thus, the following lemma is an immediate corollary of Lemma~\ref{plmp}.

\begin{lemma}[\cite{CL,DT}]
\label{dta}
If $C_i\ne C$ then the $\x$-expansion of $x_{C_i}$ does not contain a term coinciding with $x_C$.
\end{lemma} 

If we know that $z\in\A$ is positive, Lemma~\ref{dta} implies that coefficient $\lambda$ of $x_C$ in the expression for $z$ is non-negative, which implies that $\B$ is an atomic basis (as $C\in\{C_i\}$ was arbitrary).

We are now ready to prove the first theorem. The proof for algebras of type $C_n$ and $B_n$ can be formulated in terms of triangulations, and the proof for algebra $F_4$ can be formulated in terms of triangulated heptagons~\cite{Lam}, but we avoid this language to produce a uniform reasoning (algebra of type $G_2$ has rank two and its cluster monomial basis was studied in~\cite{SZ,LLZ}). 

\begin{proof}[Proof of Theorem~\ref{atomF}]
The basis $\B$ for the corresponding cluster algebra consists of cluster monomials only. To prove the theorem, we need to show that the skew-symmetrizable counterpart of Lemma~\ref{dta} holds. Namely, we want to prove that a Laurent expansion of a cluster monomial does not have other summands being cluster monomials. As in the skew-symmetric case, this proves the atomicity of the basis $\B$.

We prove the counterpart of Lemma~\ref{dta} by contradiction. Suppose it fails, so there exist a cluster $\x$ of $\A$, a cluster monomial $x_C$ in $\x$, and a cluster monomial $x_{C_0}$ containing at least one variable not compatible with factors of $x_C$ such that the $\x$-expansion of  $x_{C_0}$ contains a term $x_C$.  

Consider an unfolding $\t\A$ of $\A$, and denote by $\t\x$ the cluster of $\t\A$ which gives $\x$ after the specialization of variables. By the construction of the unfolding, cluster monomials lift to cluster monomials with literally the same expansions (after identifications of the corresponding variables). Thus, the $\t\x$-expansion of some lift $\t x_{C_0}$ of $x_{C_0}$  contains (after the specialization of variables) the term $x_C$. Our goal is to show that this contradicts Lemma~\ref{plmp}.

Let $\{x_i\}$ be cluster variables of $\x$, and denote $\{\t x_{i,s}\}_{s\in\{1,2\}}$ the lifts of $\{x_i\}$ (some of $x_i$ have only one lift, we will denote these by $\{\t x_i\}$). Denote by $\{u_i\}$ the factors of $x_{C_0}$, and by $\{\t u_{i,s}\}_{s\in\{1,2\}}$ the lifts of $\{u_i\}$ (again, we will denote by $\t u_i$ those who are the unique lifts).

Denote by $\t x_{C_0}^1$ any lift of $x_{C_0}$
(note that $\t x_{C_0}^1$ is a cluster monomial of $\t\A$). Take the term in the $\t \x$-expansion of $\t x_{C_0}^1$ which becomes $x_C$ after specialization of variables, this term should have the form
$$\prod\limits_{j\in I}\t x_j^{r_j}\prod\limits_{i\notin I}\frac{\t x_{i,s_i}^{k_i}}{\t x_{i,t_i}^{m_i}},$$
where $I$ is the index set of variables of $\x$ with a unique lift, $s_i\ne t_i$, $r_j\ge 0$, and $k_i\ge m_i\ge 0$.

Now consider another lift $\t x_{C_0}^2$ of the cluster monomial $x_{C_0}$ obtained from $\t x_{C_0}^1$ in the following way: for every factor $u_i$ we swap all the lifts $\t u_{i,1}$ and $\t u_{i,2}$. Then we obtain a cluster monomial of $\t\A$ compatible with $\t x_{C_0}^1$, and its $\t\x$-expansion can be obtained from  the $\t \x$-expansion of $\t x_{C_0}^1$ by swapping all the $\t x_{i,1}$ and $\t x_{i,2}$ (this can be easily seen for individual cluster variables $\t u_{i,1}$ and $\t u_{i,2}$ by, e.g., using cluster automorphisms technique~\cite{Law}, and thus holds for cluster monomials as well). Therefore, the $\t \x$-expansion of $\t x_{C_0}^2$ has a term 
$$\prod\limits_{j\in I}\t x_j^{r_j}\prod\limits_{i\notin I}\frac{\t x_{i,t_i}^{k_i}}{\t x_{i,s_i}^{m_i}}.$$

Since $\t x_{C_0}^1$ and $\t x_{C_0}^2$ are cluster monomials in the same cluster of $\t\A$, their product is also a cluster monomial, and its $\t\x$-expansion contains the product of the two terms above, which is not a proper Laurent monomial since $k_i\ge m_i$. This contradicts Lemma~\ref{plmp}.

\end{proof}

For the $\t A_{p,q}$ case the situation is a bit more involved. Let $\B$ be the bracelet basis of a cluster algebra $\A$ of type $A_{p,q}$, take an element $y=\sum_i\lambda_ix_{C_i}$, where $x_{C_i}\in\B$. Choose any $C\in\{C_i\}$. Let $T$ be any triangulation of an annulus containing all the arcs from the multicurve $C$ (note that $C$ may also contain a bracelet). Now consider the $T$-expansions of all $x_{C_i}$. 

This time $x_{C}$ is either  a monomial in these variables (if $C$ contained arcs only), or a sum of Laurent monomials with positive coefficients (if $C$ contains a bracelet). In the latter case, there are infinitely many triangulations containing all the arcs from the multicurve $C$. 

\begin{lemma}[\cite{DT}]
\label{dtat}
(i) If $C_i$ is not compatible with $T$, and $C$ contains no bracelets, then the $T$-expansion of $x_{C_i}$ is a proper Laurent monomial. In particular, the $T$-expansion of $x_{C_i}$ does not contain any term coinciding with $x_{C}$.

(ii) Let $C$ contain a bracelet. If $C_i\ne C$, then there exist a sequence of triangulations $T_r$ such that each $T_r$ contains all arcs of $C$,  $r_0\in\N$, and a term $x^{\u C,r}$ of $T_r$-expansion of $x_C$ such that for any $r>r_0$ the $T_r$-expansion of $x_{C_i}$ does not contain $x^{\u C,r}$.
\end{lemma} 
We will recall the construction of $T_r$ and $x^{\u C,r}$ during the proof.

As in the finite type case, Lemma~\ref{dtat} implies that $\B$ is an atomic basis.


\begin{proof}[Proof of Theorem~\ref{atom}]

Similarly to the proof of Theorem~\ref{atomF}, we now need a counterpart of Lemma~\ref{dtat} to hold. The corresponding orbifold is  a disc with two orbifold points and $n$ marked points, the bracelet basis for the corresponding cluster algebra contains cluster monomials and one family of bracelets.  If we take a double cover of the disc branching in orbifold points only, we obtain an annulus with $n$ marked points on each boundary component, so that the unfolding is $\t A_{n,n}$. Let $x_C$ be an element of the basis $\B$. 

If $C$ does not contain bracelets, then the argument is similar to one for $C_n$. Take any basis element $x_{C_0}$ and any triangulation $T$ containing all the arcs of $C$. Let $\{\gamma_i\}$ be the set of all arcs in $T$, and let $\{\t \gamma_{i,s}\}_{s\in\{1,2\}}$ be the lifts of $\{\gamma_i\}$ (the two pending arcs of $T$ have only one lift each). Choose any lift ${\t C_0}$ of ${C_0}$, take any term $\t x^t$ of $\t T$-expansion of $x_{\t C_0}$. By Lemma~\ref{dtat}(i), this term has negative degrees with respect to some $\{\t\gamma_{i,1}\}_{i\in I_1}$ and $\{\t\gamma_{i,2}\}_{i\in I_2}$. If we assume that after identification of ${\t \gamma_{i,1}}$ and ${\t \gamma_{i,2}}$ the term $\t x^t$ will coincide with $x_C$ then it should become a monomial, which implies that every degree of $\t x^t$ with respect to $\{\t\gamma_{i,2}\}_{i\in I_1}$ and $\{\t\gamma_{i,1}\}_{i\in I_2}$ have to be positive and greater or equal to the moduli of the corresponding negative ones. Assume that this is the case.

First, assume that $x_{C_0}$ is a cluster monomial. Then we can consider the multicurve $\t C_0'$ symmetric to $\t C_0$, i.e. we substitute all the entries of ${\t \gamma_{i,1}}$ by ${\t \gamma_{i,2}}$ and vice versa for every $i$. Note that  $\t C_0'$ is compatible with  $\t C_0$, and thus, we can consider a cluster monomial $x_{\t C_0\cup\t C_0'}$ and the term of its $\t T$-expansion obtained as the product of $\t x^t$ and $\t x^{t'}$, where $\t x^{t'}$ is the term symmetric to  $\t x^{t}$. By the assumption, the term $\t x^t\t x^{t'}$ of the $\t T$-expansion of this monomial will have positive degree with respect to any variable, which contradicts Lemma~\ref{dtat}(i).

Now, assume that ${C_0}$ contains a bracelet $Brac_k(\gamma)$, where $\gamma$ is the semi-closed loop joining two orbifold points. Then $\t C_0$ contains a bracelet  $Brac_k(\t\gamma)$. Note that every term of the $\t T$-expansion of $x_{Brac_k(\t\gamma)}^2$ can be considered as some term of the $\t T$-expansion of $x_{Brac_{2k}(\t\gamma)}$. Thus, we can also consider $x_{\t C_0\cup\t C_0'}$: although it is not an element of the bracelet basis of $\t A_{n,n}$, all the terms of its $\t T$-expansion are terms of the $\t T$-expansion of an element of the bracelet basis  obtained from  $x_{\t C_0\cup\t C_0'}$ by dividing by $x_{Brac_k(\t\gamma)}^2$ and multiplying by $x_{Brac_{2k}(\t\gamma)}$. Considering the same term as in the previous case, we see that it has positive degree with respect to any variable, which contradicts Lemma~\ref{dtat}(i).      

Assume now that $C$ consists of a bracelet $Brac_m(\gamma)$ and a collection of arcs $\u C$. Following~\cite{DT}, define triangulations $T_r$ containing all the arcs of $C$ as follows. There exists a marked point (call it $O$) which can be joined with both orbifold points without intersecting arcs of $C$. Add to arcs of $C$ two shortest pending arcs ending at $O$ and an arc with both ends at $O$, then extend this to any triangulation, this will be $T_0$. To obtain $T_r$, apply a half of the Dehn twist to the loop around two orbifold points $r$ times. Denote the pending arcs by $\alpha$ and $\beta$. An easy computation shows that the $T$-expansion of $x_{Brac_m(\gamma)}$ has a term $x_{\beta}^m/x_{\alpha}^m$ (see also~\cite{DT}). Consider the term $x^{\u C,r}=x_{\u C}x_{\beta}^m/x_{\alpha}^m$ of the $T_r$-expansion of $x_C$. The lifts $\t T_r$ and $x^{\t{\u{C}},r}$ on the annulus are exactly the triangulation and the term used in Lemma~\ref{dtat}(ii), see~\cite{DT}. We now want to show that given $C_0$, for $r$ large enough no term of the $T_r$-expansion of $x_{C_0}$ coincides with $x^{\u C,r}$. Since any element of cluster algebra is a finite linear combination of elements of $\B$, this will complete the proof.

We proceed in exactly the same way as in the previous case. Take any lift $\t C_0$ of $C_0$ and assume that some term of the $\t T_r$-expansion of $x_{\t C_0}$ coincides with  $x^{\u C,r}$ after specialization of variables. As before, consider the  multicurve $\t C_0'$ symmetric to $\t C_0$ and the corresponding term of $x_{\t C_0'}$, take the product of these two terms: the result is a term of  the $\t T_r$-expansion of an element of the bracelet basis of $\t A_{n,n}$ obtained from  $x_{\t C_0\cup\t C_0'}$ by dividing by $x_{Brac_m(\t\gamma)}^2$ and multiplying by $x_{Brac_{2m}(\t\gamma)}$. By the assumption, this will have a form $x_{\t{\u{C}}\cup \t{\u{C}}}x_{\beta}^{2m}/x_{\alpha}^{2m}$ for some lift $\t{\u{C}}$ of $\u{C}$. Thus, we see that $\t T_r$-expansion of some element of the bracelet basis of $\t A_{n,n}$ contains a term coinciding with $x^{\t{\u C'},r}$, where $\t C'=\t{\u C}\cup\t{\u C}\cup Brac_{2m}(\t\gamma)$. According to Lemma~\ref{dtat}(ii), this cannot hold for $r$ large enough. This completes the proof of the theorem.       

\end{proof}


\begin{thebibliography}{35}


\bibitem[CK]{CK} P.~Caldero, B.~Keller, {\em From triangulated categories to cluster algebras}, Invent. Math. 172 (2008), 169--211.

\bibitem[CLS]{CLS} I.~Canakci, K.~Lee, R.~Schiffler, {\em On cluster algebras from unpunctured surfaces with one marked point}, Proc. Amer. Math. Soc. Ser. B 2 (2015), 35--49.

\bibitem[CT]{CT} I.~Canakci, P.~Tumarkin, {\em Bases of cluster algebras from unpunctured orbifolds with one marked point}, in preparation.

\bibitem[C1]{C1} G.~Cerulli~Irelli, {\em Cluster algebras of type $A_2^{(1)}$}, Algebr. Represent. Theory 15 (2012), 977--1021.

\bibitem[C2]{C} G.~Cerulli~Irelli, {\em Positivity in skew-symmetric cluster algebras of finite type}, arXiv:1102.3050.

\bibitem[CE]{CE} G.~Cerulli~Irelli, F.~Esposito, {\em Geometry of quiver Grassmannians of Kronecker type and applications to cluster algebras},  Algebra Number Theory 5 (2011), 777--801.

\bibitem[CKLP]{CKLP} G.~Cerulli~Irelli, B.~Keller, D.~Labardini-Fragoso, P.-G.~Plamondon, {\em Linear independence of cluster monomials for skew-symmetric cluster algebras}, Compos. Math. 149 (2013), 1753--1764. 

\bibitem[CL]{CL} G.~Cerulli~Irelli, D.~Labardini-Fragoso, {\em Quivers with potentials associated to triangulated surfaces, part III: Tagged triangulations and cluster monomials}, Compos. Math. 148 (2012), 1833--1866.

\bibitem[Ch]{Ch} L.~Chekhov, {\em Orbifold Riemann surfaces and geodesic algebras}, J. Phys. A: Math. Theor. 42 (2009), 304007, 32 pp.

\bibitem[ChM]{ChM} L.~Chekhov, M.~Mazzocco, {\em Orbifold Riemann surfaces: Teichm\"uller spaces and algebras of geodesic functions}, Russian Math. Surveys 64 (2009), 1079--1130.


\bibitem[CGMMRSW]{CGMMRSW} M.~W.~Cheung, M.~Gross, G.~Muller, G.~Musiker, D.~Rupel, S.~Stella, H.~Williams, {\em The greedy basis equals the theta basis}, J. Combin. Theory A, 145 (2017), 150--171.


\bibitem[DXX]{DXX} M.~Ding, J.~Xiao, F.~Xu, {\em Integral bases of cluster algebras and representations of tame quivers},  Algebr. Represent. Theory 16 (2013), 491--525.

\bibitem[D1]{D1} G.~Dupont, {\em Generic variables in acyclic cluster algebras and bases in affine cluster algebras}, arXiv:0811.2909.

\bibitem[D2]{D} G.~Dupont, {\em Transverse quiver Grassmannians and bases in affine cluster algebras}, Algebra Number Theory 4 (2010), 599--624.

\bibitem[D3]{D3} G.~Dupont, {\em Generic variables in acyclic cluster algebras}, J. Pure Appl. Algebra 215 (2011), 628--641.

\bibitem[DT]{DT} G.~Dupont, H.~Thomas, {\em Atomic bases in cluster algebras of types $A$ and $\t A$}, Proc. London Math. Soc. 107 (2013), 825--850.


\bibitem[FeSTu1]{FeSTu1} A.~Felikson, M.~Shapiro, P.~Tumarkin, {\em Skew-symmetric cluster algebras of finite mutation type}, J. Eur. Math. Soc. 14 (2012), 1135--1180.

\bibitem[FeSTu2]{FeSTu2}  A.~Felikson, M.~Shapiro, P.~Tumarkin, {\em Cluster algebras of finite mutation type via unfoldings},  Int. Math. Res. Notices (2012), 1768--1804.

\bibitem[FeSTu3]{FeSTu3}  A.~Felikson, M.~Shapiro, P.~Tumarkin, {\em Cluster algebras and triangulated orbifolds}, Adv. Math. 231 (2012), 2953--3002.


\bibitem[FG1]{FG0} V.~Fock, A.~Goncharov, {\em Moduli spaces of local systems and higher Teichm\"uller theory}, Publ. Math. Inst. Hautes \'Etudes Sci. 103 (2006), 1--211.

\bibitem[FG2]{FGl} V.~Fock, A.~Goncharov, {\em Dual Teichm\"uller and lamination spaces}. Handbook of Teichm\"uller theory. Vol. I, 647--684, IRMA Lect. Math. Theor. Phys., 11, Eur.Math. Soc., Z\"urich, 2007.

\bibitem[FG3]{FG}  V.~Fock, A.~Goncharov, {\em Cluster ensembles, quantization and the dilogarithm}, Ann. Sci. \'Ec. Norm. Sup\'er. 42 (2009), 865--930.



\bibitem[FST]{FST} S.~Fomin, M.~Shapiro, D.~Thurston, {\em Cluster algebras and triangulated surfaces. Part {\rm I}: Cluster complexes}, Acta Math. 201 (2008), 83--146.

\bibitem[FT]{FT} S.~Fomin, D.~Thurston, {\em Cluster algebras and triangulated surfaces. Part {\rm II}: Lambda lengths}, arXiv:1210.5569.

\bibitem[FZ1]{FZ1} S.~Fomin, A.~Zelevinsky, {\em Cluster algebras {\rm I}: Foundations}, J. Amer. Math. Soc. 15 (2002), 497--529.



\bibitem[FZ2]{FZ4} S.~Fomin, A.~Zelevinsky, {\em Cluster algebras {\rm IV}: Coefficients}, Compos. Math. 143 (2007), 112--164.

\bibitem[GHKK]{GHKK} M.~Gross, P.~Hacking, S.~Keel, M.~Kontsevich, {\em Canonical bases for cluster algebras}, arXiv:1411.1394 

\bibitem[GLS]{GLS} C.~Geiss, B.~Leclerc,  J.~Schr\"oer, {\em Generic bases for cluster algebras and the Chamber Ansatz}, J. Amer. Math. Soc. 25 (2012), 21--76.

\bibitem[GM]{GM} E.~Gunawan, G.~Musiker, {\em $T$-path formula and atomic bases for cluster algebras of type $D$}, SIGMA 11 (2015), 060, 46pp.


\bibitem[K]{K} M.~Kashiwara, {\em Bases cristallines}, C. R. Acad. Sci. Paris S\'er. I Math. 311 (1990), 277--280.

\bibitem[Lam]{Lam} L.~Lamberti, {\em Combinatorial model for the cluster categories of type $E$}, J. Alg. Combin. 41 (2015), 1023--1054.

\bibitem[Law]{Law} J.~Lawson, {\em Cluster automorphisms and the marked exchange graphs of skew-symmetrizable cluster algebras}, Electron. J. Combin. 23 (2016), Paper \#P4.41.

\bibitem[LLZ]{LLZ} K.~Lee, L.~Li, A.~Zelevinsky, {\em Greedy elements in rank $2$ cluster algebras}, Selecta Math. 20 (2014), 57--82.

\bibitem[L]{L} G.~Lusztig, {\em Canonical bases arising from quantized enveloping algebras}, J. Amer. Math. Soc. 3 (1990), 447--498.


\bibitem[MSW1]{MSW} G.~Musiker, R.~Schiffler, L.~Williams, {\em Positivity for cluster algebras from surfaces}, Adv. Math. 227 (2011), 2241--2308.

\bibitem[MSW2]{MSW2} G.~Musiker, R.~Schiffler, L.~Williams, {\em Bases for cluster algebras form surfaces}, Compos. Math 149, 2, (2013) 217-263.

\bibitem[MW]{MW}  G.~Musiker, L.~Williams, {\em
Matrix Formulae and skein relations for cluster algebras from surfaces}, Int. Math. Res. Notices (2013), 2891--2944. 

\bibitem[NZ]{NZ} T.~Nakanishi, A.~Zelevinsky, {\em On tropical dualities in cluster algebras}, Contemp. Math. 565 (2012) 217--226

\bibitem[P]{Pe} R.~C.~Penner, {\em The decorated Teichm\"uller space of punctured surfaces}, Comm. Math. Phys. {\bf 113} (1987), 299--339.

\bibitem[Pl]{P} P.~G.~Plamondon, {\em Generic bases for cluster algebras from the cluster category}, Int. Math. Res. Notices (2013), 2368--2420.


\bibitem[SZ]{SZ} P.~Sherman, A.~Zelevinsky, {\em Positivity and canonical bases in rank $2$ cluster algebras of finite and affine types}, Mosc. Math. J. 4 (2004), 947--974.

\bibitem[T1]{T} D.~Thurston, {\em The Geometry and Algebra of Curves on Surfaces}, Lecture course in UC Berkeley, 2012. Notes by Qiaochu Yuan 
 \url{http://math.berkeley.edu/~qchu/Notes/274/}

\bibitem[T2]{T2} D.~Thurston, {\em Positive basis for surface skein algebras}, Proc. Natl. Acad. Sci. USA 111 (2014), 9725--9732.

\bibitem[T]{Th} W. P. Thurston, {\em On the geometry and dynamics of diffeomorphisms of surfaces}, Bull. Amer. Math. Soc. (N.S.) 19 (1988), 417--431.

\end{thebibliography}
\end{document}